\documentclass[11pt,reqno]{amsart}

\usepackage{enumerate, amsmath, amsfonts, amssymb, amsthm, mathtools, thmtools, wasysym, graphics, graphicx, xcolor,frcursive,xparse,comment,bbm,young,array,tabularx,combelow}
\usepackage[all]{xy}
\usepackage[colorlinks=true, pdfstartview=FitV, linkcolor=blue, citecolor=blue, urlcolor=blue]{hyperref}

\usepackage{url, hypcap}
\hypersetup{colorlinks=true, citecolor=darkblue, linkcolor=darkblue}
\usepackage{mathrsfs}
\usepackage{booktabs,tabmac}

\usepackage{tikz}
\usetikzlibrary{arrows,automata}
\usetikzlibrary{patterns,snakes}
\usetikzlibrary{shapes.geometric,calc}
\tikzstyle{arrow}=[thick, ->, >=stealth]
\tikzset{
	edge/.style={->,> = latex'}
}
\usetikzlibrary{calc,through,backgrounds,shapes,matrix}
\usepackage{graphicx}

\usepackage[utf8]{inputenc}

\usepackage[letterpaper,margin=1.25in]{geometry}

\definecolor{darkblue}{rgb}{0.0,0,0.7} % darkblue color
\definecolor{lightblue}{rgb}{0,135,147}

\definecolor{red}{rgb}{1,0,0}

\definecolor{darkred}{rgb}{0.7,0,0} % darkred color
\definecolor{lightgrey}{rgb}{0.7,0.7,0.7} % darkred color
\newcommand{\defn}[1]{{\color{darkred}\emph{#1}}} % emphasis of a definition
\newcommand{\mms}{_}
\newcommand{\area}{\mathsf{area}}
\newcommand{\dinv}{\mathsf{dinv}}
\newcommand{\stat}{\mathsf{stat}}
\newcommand{\comb}{\operatorname{comb}}

\newcommand{\wt}{\mathrm{wt}}
\newcommand{\twt}{\widetilde{\wt}}

\newtheorem{theorem}{Theorem}[section]
\newtheorem{lemma}[theorem]{Lemma}
\newtheorem{conjecture}[theorem]{Conjecture}
\newtheorem{corollary}[theorem]{Corollary}
\newtheorem{proposition}[theorem]{Proposition}
\theoremstyle{definition}

\newtheorem{example}[theorem]{Example}
\newtheorem{remark}[theorem]{Remark}
\numberwithin{equation}{section}

\usepackage[colorinlistoftodos]{todonotes}

\DeclareMathOperator{\FHilb}{FHilb}

\newcommand{\CC}{\mathbb{C}}
\newcommand{\CL}{\mathcal{L}}
\newcommand{\CO}{\mathcal{O}}
\DeclareMathOperator{\Res}{Res}

\title{Generalized $q,t$-Catalan numbers}

\author[Gorsky]{Eugene Gorsky}
\address[E. Gorsky]{Department of Mathematics, University of California, One Shields
Avenue, Davis, CA 95616-8633, U.S.A.}
\address[E. Gorsky]{International Laboratory of Representation Theory and Mathematical Physics, NRU-HSE, Moscow, Russia}
\email{egorskiy@math.ucdavis.edu}
\urladdr{http://www.math.ucdavis.edu/\~{}egorskiy}

\author[Hawkes]{Graham Hawkes}
\address[G. Hawkes]{Department of Mathematics, University of California, One Shields
Avenue, Davis, CA 95616-8633, U.S.A.}
\curraddr{Max--Planck--Institut f\"ur Mathematik, Vivatsgasse 7, 53111 Bonn, Germany}
\email{hawkes@math.ucdavis.edu}
\urladdr{http://guests.mpim-bonn.mpg.de/hawkes/}

\author[Schilling]{Anne Schilling}
\address[A. Schilling]{Department of Mathematics, University of California, One Shields
Avenue, Davis, CA 95616-8633, U.S.A.}
\email{anne@math.ucdavis.edu}
\urladdr{http://www.math.ucdavis.edu/\~{}anne}

\author[Rainbolt]{Julianne Rainbolt}
\address[J. Rainbolt]{Department of Mathematics and Statistics, Saint Louis University, 220 North Grand Blvd., 
Saint Louis, MO 63103, U.S.A.}
\email{rainbolt@slu.edu}
\urladdr{https://mathstat.slu.edu/people/rainbolt}

\begin{document}

\begin{abstract}
Recent work of the first author, Negu\cb{t} and Rasmussen, and of Oblomkov and Rozansky in the context of
Khovanov--Rozansky knot homology produces a family of polynomials in $q$ and $t$ labeled by integer sequences. 
These polynomials can be expressed as equivariant Euler characteristics of certain line bundles on flag Hilbert schemes.
The $q,t$-Catalan numbers and their rational analogues are special cases of this construction. In this paper, we give a 
purely combinatorial treatment of these polynomials and show that in many cases they have nonnegative integer coefficients. 

For sequences of length at most 4, we prove that these coefficients enumerate subdiagrams in a certain fixed Young 
diagram and give an explicit symmetric chain decomposition of the set of such diagrams. This strengthens results of 
Lee, Li and Loehr for $(4,n)$ rational $q,t$-Catalan numbers.
\end{abstract}

\maketitle

%%%%%%%%%%%%%%%%%%%%%%%%%%%%%%%%%%%%%%%%%%%%%%%%%%%%%%%%%
\section{Introduction}
%%%%%%%%%%%%%%%%%%%%%%%%%%%%%%%%%%%%%%%%%%%%%%%%%%%%%%%%%

The last decade revealed deep, and yet partially conjectural 
connections~\cite{GORS.2014,GN.2015,GM.2013,GM.2014,GM.2016,GMV.2016,GMV.2017}
of the HOMFLY-PT link homologies with various intricate constructions in algebraic combinatorics such as $q,t$-Catalan 
numbers of Garsia and Haiman~\cite{GarsiaHaiman.1996}, LLT polynomials \cite{HHLRU.2005}, and the elliptic Hall 
algebra~\cite{SV.2013}. Some of these conjectures were recently proven (mostly for the torus knots and links) by Elias,
Hogancamp and Mellit~\cite{EH.2019,Hogancamp.2017,Mellit.2017}. 

An interesting class of knots, which best fits in the framework of the above conjectures, are the so-called Coxeter links 
defined as closures of braids 
\[
	\beta(a_1,\ldots,a_n)=\ell_1^{a_1}\cdots \ell_n^{a_n}t_1\cdots t_{n-1},
\]
where $\ell_i=t_{i-1}\cdots t_1t_1\cdots t_{i-1}$ are Jucys--Murphy elements and $t_i$ are the standard braid group 
generators. Here $a_i$ are arbitrary integers, but in this paper we will mostly assume $a_i\geqslant 0$, so that all crossings 
in the braid $\beta(a_1,\ldots,a_n)$ are positive.
 
Motivated by the geometry of the flag Hilbert scheme of points on the plane (see Section~\ref{subsec.fhilb} and 
references therein) we can approximate the invariants of such knots with the following combinatorial expressions. Define
\begin{equation}
\label{eq: def f}
	f(a_1,\ldots,a_n)=\sum_{T}z_1^{a_1}\cdots z_n^{a_n} \; \prod_{i=2}^{n} \frac{1}{(1-z_i^{-1})(1-qtz_{i-1}/z_i)}
	\; \prod_{i<j}\omega(z_i/z_j), 
\end{equation}
where the sum is over standard tableaux $T$ with $n$ boxes, $z_i$ is the $(q,t)$-content $q^{c-1} t^{r-1}$
of the box labeled by $i$ in row $r$ and column $c$ in $T$, and $\omega(x)=\frac{(1-x)(1-qtx)}{(1-qx)(1-tx)}$. A priori, 
this is a rational function in $q$ and $t$, but we prove 
in Section~\ref{subsec:Tesler} that it is always a polynomial in $q$ and $t$ with integer coefficients. This polynomial can 
be expressed as a sum over Tesler matrices with row sums $a_i$ as in  \cite{GN.2015} and 
especially~\cite{AGHRS.2012}, where similar polynomials have already appeared. 

In the special case when 
\[
	a_i=S_{i}(m,n):=  \left \lceil \frac{im}{n} \right \rceil - \left \lceil \frac{(i-1)m}{n} \right \rceil,
\]
by \cite{GN.2015} the function $f(S_{i}(m,n))$ agrees with the \defn{rational $q,t$-Catalan number} $c_{m,n}(q,t)$.\footnote{Note that the formula for $S_i(m,n)$ in \cite{GN.2015} used floors instead of ceilings, but the two are related by the change $i\to n+1-i$. This change is implicit in \cite{GN.2015} since that paper uses opposite conventions for standard tableaux.}
By the main result of \cite{Mellit.2016}, this is a polynomial in $q$ and $t$ with nonnegative coefficients. More precisely, 
\[
	c_{m,n}(q,t)=\sum_{D}q^{\area(D)}t^{\dinv(D)},
\]
where the sum is over all Dyck paths $D$ in the $m\times n$ rectangle and $\area(D)$, $\dinv(D)$ are certain 
combinatorial statistics (see for example~\cite{Haglund.2008}). In the even more special case $m=n+1$, we obtain 
$a_i=S_i(n+1,n)=1$ for $i>1$, and the polynomial $f(1,\ldots,1)=f(2,1,\ldots,1)$ agrees with the
\defn{$q,t$--Catalan number} of Garsia and Haiman~\cite{GarsiaHaiman.1996}. 

Motivated by~\cite{GNR.2016,OR.2017}, we expect that the beautiful combinatorics of $q,t$-Catalan numbers and 
their rational analogues can be generalized to the case of arbitrary $a_i$, possibly constrained by some inequalities. 
In fact, as we show in this paper, that varying $a_i$ allows one to compute the invariants $f(a_1,\ldots,a_n)$ recursively, 
see Corollary~\ref{corollary.recursion} for the $n=4$ example. 

Using the machinery of Tesler matrices, we  prove the following result.
\begin{proposition}
Suppose that $a_i\geqslant 0$. Then $f(a_1,\ldots,a_n)$ is a polynomial in $q$ and $t$.
At $t=1$, this polynomial specializes to 
\[
	f(a_1,\ldots,a_n) \Big|_{t=1} 
	=\sum_{\mu\subseteq \lambda(a)}q^{|\lambda(a)|-|\mu|},
\]
where $\lambda(a)=(a_2+\cdots+a_n,a_3+\cdots+a_n,\ldots,a_n).$
\end{proposition}

\begin{example}
\label{example.2 and 3}
For $n=2$, one has 
\[
	f(a_1,a_2)=[a_2+1]_{q,t}:=q^{a_2}+q^{a_2-1}t+\cdots+qt^{a_2-1}+t^{a_2}.
\]
For $n=3$ and $a_2\geqslant a_3$ one has
\[
	f(a_1,a_2,a_3)=[a_2+2a_3+1]_{q,t}+qt[a_2+2a_3-2]_{q,t}+\cdots+q^{a_3}t^{a_3}[a_2-a_3+1]_{q,t}.
\]
See Examples~\ref{example.n=2} and~\ref{example.n=3} for derivations of these formulas.
\end{example}

The following conjecture was communicated to the authors by Andrei Negu\cb{t}.

\begin{conjecture}[Negu\cb{t}]
\label{conj:positivity}
If $a_1\geqslant a_2 \geqslant \cdots\geqslant a_n\geqslant 0$, then $f(a_1,\ldots,a_n)$ is a polynomial in $q$ and $t$ 
with nonnegative coefficients.
\end{conjecture}

For general $a_1\geqslant a_2 \geqslant \cdots\geqslant a_n\geqslant 0$, it is still an open problem to find an explicit 
statistic $\stat$ on partitions $\mu$ such that
\begin{equation}
\label{eq: mystery stat}
	f(a_1,\ldots,a_n)=\sum_{\mu\subseteq \lambda(a)}q^{|\lambda(a)|-|\mu|}\; t^{\stat(\mu)}.
\end{equation}
In this paper, we solve the problem for $n=4$:

\begin{theorem}
For $a+1 \geqslant b$, $a+1,b+1 \geqslant c\geqslant 0$, the polynomial $F(a,b,c) := f(a_1,a,b,c)$ has nonnegative 
integer coefficients and can be written in the form~\eqref{eq: mystery stat}. 
The statistic $\stat(\mu)$ arises from an explicit decomposition of the set of $\mu\subseteq \lambda(a)$
into symmetric chains. 
\end{theorem}
See Section~\ref{section.combinatorial expressions} for further details.

Since a symmetric chain specializes to $q^{k}+q^{k-2}+\cdots+q^{-k+2}+q^{-k}$ a $t=q^{-1}$, we immediately obtain the 
following corollary.

\begin{corollary}
For $a+1 \geqslant b$, $a+1,b+1 \geqslant c\geqslant 0$, the coefficients of the specialization
$F(a,b,c)|_{t=q^{-1}}$ are unimodular in even and in odd degrees. 
\end{corollary}

\begin{remark}
By \cite{GNR.2016,OR.2017} the specialization of $f(a_1,\ldots,a_n)$ at $q=t^{-1}$ coincides with the  part of the 
HOMFLYPT polynomial of the knot $\beta(a_1,\ldots,a_n)$. 
\end{remark}

\begin{remark}
Our statistic and decomposition is different from that in~\cite{LLL.2014, LLL.2018}. In particular, some of their chains 
are not symmetric, but the authors show that partitions come in symmetric pairs. 
\end{remark}

We provide a recursion for $F(a,b,c)$ and prove that the combinatorial expression also satisfies the recursion
(see Sections~\ref{subsec:recursion n=4} and~\ref{section.combinatorial recursion}).

The set of Young diagrams $\mu$ contained in the diagram $\lambda(a)$ is in bijection with the Demazure 
crystal~\cite{Kashiwara.1993,Littelmann.1995} with highest weight $(a_1,\ldots,a_n)$ and Weyl group element 
$c=t_1\cdots t_{n-1}$. The size of $\mu$ can be easily expressed in terms of the weight of the corresponding element 
of the crystal basis. This observation leads to many interesting questions:
\begin{itemize}
\item What is the crystal-theoretic interpretation of the statistic $\stat$? 
\item Is there a crystal-theoretic interpretation of the symmetric chains and the polynomials $f(a_1,\ldots,a_n)$?
\end{itemize}

\begin{remark}
In the terminology of~\cite{CDL.2018}, subdiagrams of $\lambda(a)$ correspond to so-called $s$-Dyck paths, and it is 
shown in \cite{CDL.2018} that they are in bijection with remarkably many combinatorial objects, just as usual Catalan 
numbers are in bijection with trees, triangulations etc. It would be interesting to relate the results of \cite{CDL.2018} 
both to Demazure crystals and to the above statistic  $\stat$.
\end{remark}

The paper is organized as follows. In Section \ref{sec:algebraic}, we discuss the algebraic aspects of
the function $f(a_1,\ldots,a_n)$ (or equivalently $F(a_2,\ldots,a_n)$). The definition of the function $f(a_1,\ldots,a_n)$ 
is given in Section~\ref{subsection.formula}. In Section \ref{subsec.fhilb}, we briefly recall its connection to flag 
Hilbert schemes and knot invariants; combinatorially inclined readers are welcome to skip this section. In 
Section~\ref{subsec:Tesler}, we connect $f(a_1,\ldots,a_n)$ to Tesler matrices and prove that they are indeed 
polynomials in $q$ and $t$. In Section~\ref{subsec:recursion n=4} we prove the recursion for $n=4$.
Section~\ref{section.combinatorial expressions} contains the combinatorial expressions for $F(a,b,c)$. We also
provide examples. In Section~\ref{section.proof}, we construct the symmetric chains underlying the combinatorial
formulas explicitly and also prove the combinatorial formulas.

%%%%%%%%%%%%%%%%%%%%%%%%%%%%%%%%%%%%%%%%%%%%%%%%%%%%%%%%%
\subsection*{Acknowledgments}
The authors would like to thank the American Institute for Mathematics (AIM) for hosting the conference
``Categorified Hecke algebras, link homology, and Hilbert schemes'' in October 1-5, 2018, where this work began.
E. G. would also like to thank Andrei Negu\cb{t} and Karola Meszaros for very useful and inspiring discussions. 

E.G. was partially supported by NSF grants DMS--1700814, DMS--1760329, and the Russian Academic Excellence 
Project 5-100. A.S. was partially supported by NSF grants DMS--1760329 and DMS--1764153.  

%%%%%%%%%%%%%%%%%%%%%%%%%%%%%%%%%%%%%%%%%%%%%%%%%%%%%%%%%
\section{The algebraic side}
\label{sec:algebraic}
%%%%%%%%%%%%%%%%%%%%%%%%%%%%%%%%%%%%%%%%%%%%%%%%%%%%%%%%%

%%%%%%%%%%%%%%%%%%%%%%%%%%%%%%%%%%%%%%%%%%%%%%%%%%%%%%%%%
\subsection{The formula}
\label{subsection.formula}

Given a standard tableau $T$ of size $n$, we define a vector $z(T)=(z_i)_{1\leqslant i \leqslant n}$, where $z_i$ is the 
$(q,t)$-content of the box in $T$ labeled by $i$. The \defn{$(q,t)$-content} of the box with row and column coordinates 
$(r,c)$, is $q^{c-1} t^{r-1}$. For example, for the tableau
$$
	T=\tableau[scY]{
	4 \cr
	3 & 5\cr
	1& 2 & 6 & 7\cr}
$$
we have
$$
	z(T)=(1,q,t,t^2,qt,q^2,q^3).
$$
By convention, $z_1=1$. We define the \defn{weight} of a tableau $T$ by 
$$
	\wt(T)=\wt(z(T))=\prod_{i=2}^n\frac{1}{(1-z_i^{-1})(1-qtz_{i-1}/z_i)}\prod_{i<j}
	\frac{(1-z_i/z_j)(1-qtz_i/z_j)}{(1-qz_i/z_j)(1-tz_i/z_j)}.
$$
Note that some of the individual factors in this product (both in the numerator and denominator) could vanish, and 
the convention is that we simply ignore these factors. Given a vector of integers $(a_2,\ldots,a_n)$ with $n\geqslant 2$, 
we define
\begin{equation}
\label{eq: def F}
	F(a_2,\ldots,a_n)=\sum_{T}z_2^{a_2}\cdots z_n^{a_n}\cdot \wt(T),
\end{equation}
where the summation is over all standard tableaux of size $n$.
\begin{proposition}
\label{prop: F polynomial}
For all integer vectors $(a_2,\ldots,a_n)$, the function $F(a_2,\ldots,a_n)$ is a polynomial in $q$ and $t$ with 
integer coefficients.
\end{proposition}
The proof is very similar to the computations in~\cite[Section 6.5]{GN.2015}, but we present it in Section~\ref{subsec:Tesler} 
for completeness.

\begin{remark}
For $a_2=\cdots=a_n=m$, the polynomial $F(a_2,\ldots,a_n)$ agrees with the Fuss--Catalan polynomial, 
see~\cite{GN.2015} and~\cite{Mellit.2016}.
\end{remark}

The following conjecture was communicated to the authors by Andrei Negu\cb{t}. 
 
\begin{conjecture}[Negu\cb{t}]
For $a_2\geqslant a_3\geqslant \cdots \geqslant a_n \geqslant 0$, the polynomial $F(a_2,\ldots,a_n)$ has nonnegative
coefficients.
\end{conjecture}

In this paper, we prove this conjecture for $n=2,3$ and $4$ in the slightly more general case
$a_2+1\geqslant a_3$, $a_2+1,a_3+1\geqslant a_4\geqslant 0$. In addition, we provide explicit combinatorial formulas
for $F(a_2,a_3,a_4)$ in this case (see Section~\ref{section.combinatorial expressions}).

\begin{remark}
Note that it is not enough to assume that $a_{i-1}+1 \geqslant a_i$ in the conjecture. For example, 
\begin{multline*}
	F(0,1,2)=q^8 + q^7t + q^6t^2 + q^5t^3 + q^4t^4 + q^3t^5 + q^2t^6 + qt^7 + t^8 + q^6t + q^5t^2 + q^4t^3 + q^3t^4 \\
	+ q^2t^5 + qt^6 + q^5t + 2q^4t^2 + 2q^3t^3 + 2q^2t^4 + qt^5 - q^4t - q^3t^2 - q^2t^3 - qt^4.
\end{multline*}
On can check that $F(1,2,3)$ contains negative terms as well.
\end{remark}

%%%%%%%%%%%%%%%%%%%%%%%%%%%%%%%%%%%%%%%%%%%%%%%%%%%%%%%%%%%%%%
\subsection{Flag Hilbert schemes}
\label{subsec.fhilb}

The definition of $F(a_2,\ldots,a_n)$ is motivated by the geometry of the flag Hilbert scheme of points on the plane,
which we briefly review here.

The flag Hilbert scheme $\FHilb^n(\CC^2)$ is defined as the moduli space of flags  
\[
	\FHilb^n(\CC^2)=\{\CC[x,y]=I_0\supset I_1\supset I_2 \supset \cdots \supset I_n\},
\]
where all $I_k$ are ideals in $\CC[x,y]$ of codimension $k$. Similarly, the punctual flag Hilbert scheme $\FHilb^n(\CC^2,0)$
is defined as the set of such flags, where all $I_k$ are supported at the origin.

The dilation action of $(\CC^*)^2$ on $\CC^2$ defined by $(x,y)\mapsto (q^{-1}x,t^{-1}y)$ lifts to an action on both 
$\FHilb^n(\CC^2)$ and $\FHilb^n(\CC^2,0)$. The fixed points of this action correspond to the flags of monomial ideals, 
and it is easy to see that these are in bijection with standard Young tableaux of size $n$. The flag Hilbert scheme carries 
natural line bundles $\CL_k:=I_{k-1}/I_k$ which are equivariant with respect to the action of $(\CC^*)^2$.
The weight of the line bundle $\CL_k$ at a fixed point corresponding to a standard tableau $T$ equals the 
$(q,t)$-content $z_k(T)$. Note that the line bundle $\CL_1$ is trivial.

The results and conjectures in \cite{GNR.2016,OR.2017} lead to the following conjecture.

\begin{conjecture}
For all $a_i$ the Khovanov--Rozansky homology of the closure of the braid $\beta(a_1,\ldots,a_n)$ 
(defined in the introduction) is isomorphic to the total sheaf cohomology
\[
	H^{\bullet}(\FHilb^n(\CC^2,0),\CL_1^{a_1}\otimes\cdots\otimes \CL_n^{a_n}).
\]
\end{conjecture}

For small values of $n$, the geometry of $\FHilb^n(\CC^2,0)$ can be described explicitly. For $n=2$ we have 
\[
	\FHilb^2(\CC^2,0)=\mathbb{P}^1,\ \CL_2=\CO(1),
\]
so 
\[
	H^{\bullet}(\FHilb^2(\CC^2,0),\CL_1^{a_1} \CL_2^{a_2})=H^{\bullet}(\mathbb{P}^1,\CO(a_2)).
\]
Furthermore, for $a_2\geqslant 0$ higher cohomology vanishes and the $(\CC^*)^2$-equivariant character of the space of 
global sections agrees with $F(a_2)$. 

For $n=3$ the space $\FHilb^3(\CC^2,0)$ is a smooth cubic Hirzebruch surface, 
and the line bundles $\CL_1^{a_1}\CL_2^{a_2}\CL_3^{a_3}$ and their cohomology can be described explicitly for 
all $a_1,a_2,a_3$, see \cite{GNR.2016}. 
Indeed, there is a natural projection $\pi \colon \FHilb^3(\CC^2,0)\to \FHilb^2(\CC^2,0)=\mathbb{P}^1$
and  for $a_3\geqslant 0$ one has 
\[
	\pi_*\CL_3^{a_3}=\mathrm{Sym}^{a_3}(\CO(2)\oplus\CO(-1))
	=\CO(2a_3)\oplus \CO(2a_3-3)\oplus\cdots\oplus \CO(-a_3),
\]
so
\begin{multline}
\label{eq: homology fhilb3}
	H^{\bullet}(\FHilb^3(\CC^2,0),\CL_1^{a_1}\CL_2^{a_2}\CL_3^{a_3})
	=H^{\bullet}(\mathbb{P}^1,\CO(a_2)\otimes \pi_*\CL_3^{a_3})\\
	=H^{\bullet}(\mathbb{P}^1,\CO(2a_3+a_2)\oplus\cdots \oplus \CO(a_2-a_3)).
\end{multline}
In particular, for $a_2\geqslant a_3$ higher cohomology vanishes and the $(\CC^*)^2$-equivariant character of the space 
of global sections agrees with $F(a_2,a_3)$, compare \eqref{eq: homology fhilb3} with Example \ref{example.2 and 3}. 

\begin{remark}
For $(a_2,a_3)=(0,2)$ we obtain by \eqref{eq: homology fhilb3}:
\[
	H^{\bullet}(\FHilb^3(\CC^2,0),\CL_3^{2})=H^{\bullet}(\mathbb{P}^1,\CO(4)\oplus\CO(1)\oplus\CO(-2)).
\]
Note that $H^1(\mathbb{P}^1,\CO(-2))$ is one-dimensional, which corresponds to the negative term in
\[
	F(0,2)=q^4 + q^3t + q^2t^2 + qt^3 + t^4 + q^2t + qt^2 - qt.
\]
\end{remark}

However, for $n\geqslant 4$ the spaces $\FHilb^n(\CC^2,0)$ become very singular and reducible. Still, they
carry a natural virtual structure sheaf, and one can use virtual localization techniques to prove the identity
\[
	\chi_{(\CC^*)^2}(\FHilb^n(\CC^2,0),\CL_1^{a_1}\otimes\cdots\otimes \CL_n^{a_n})=F(a_2,\ldots,a_n).
\]
Here on the left hand side, we obtain the $(\CC^*)^2$-equivariant Euler characteristic which can be computed as 
an explicit sum over fixed points of $(\CC^*)^2$ or, equivalently, over standard Young tableaux. This sum agrees 
with~\eqref{eq: def F}. We refer the reader to \cite{GN.2015} and \cite{GNR.2016} for further details.

It is important to point out that, although the polynomial $F(a_2,\ldots,a_n)$ has a geometric interpretation, this does 
not immediately imply Conjecture~\ref{conj:positivity}. Indeed, for $n=2,3$ this follows from vanishing of higher cohomology, 
but no such vanishing results are available yet for $n\geqslant 4$. It would be interesting to compare the results of this paper 
with the geometry of $\FHilb^4(\CC^2,0)$.  See \cite[Section 1.4]{GNR.2016} and \cite[Conjecture 6.4.2]{OR.2017} for more
on the geometric context.

%%%%%%%%%%%%%%%%%%%%%%%%%%%%%%%%%%%%%%%%%%%%%%%%%%%%%%%%%
\subsection{Tesler matrices}
\label{subsec:Tesler}

To prove Proposition~\ref{prop: F polynomial}, we need to use the formalism of Tesler matrices, developed 
in~\cite{Haglund.2011,AGHRS.2012,GHX.2014}. Given a sequence $a=(a_1,a_2,\ldots,a_n)$ of nonnegative integers, 
we define a \defn{Tesler matrix} to be an upper-triangular matrix $M=(m_{ij})_{j\geqslant i}$ with nonnegative integer 
coefficients $m_{ij}\geqslant 0$ satisfying a system of linear equations
\begin{equation}
\label{eq: tesler}
	m_{ii}+\sum_{j<i}m_{ji}-\sum_{j>i}m_{ij}=a_i\  \qquad \text{for}\ 1\leqslant i \leqslant n.
\end{equation}

\begin{lemma}
\label{lem:tesler finite}
The set of Tesler matrices is finite for fixed $a$.
\end{lemma}

\begin{proof}
Equation~\eqref{eq: tesler} can be rewritten as follows:
\begin{equation}
\label{eq: tesler 2}
	m_{ii}+\cdots+m_{nn}+\sum_{j<i,k\geqslant i}m_{jk}=a_i+\cdots+a_n.
\end{equation} 
Since all $m_{ij}$ are nonnegative integers, we obtain $m_{ij}\leqslant a_1+\cdots+a_n$ for all $i,j$.
\end{proof}

Given a sequence $(a_2,\ldots,a_n)$, we define a partition or Young diagram
$$\lambda(a)=(a_2+\cdots+a_n,\ldots,a_n)$$ (note that $a_1$ is not used). Let us call a Tesler matrix \defn{two-diagonal},
if $m_{ij}=0$ for $j>i+1$.

\begin{lemma}
\label{lem:subdiagrams}
There is a bijection between the set of two-diagonal Tesler matrices associated to $a=(a_1,\ldots,a_n)$ and the set 
of subdiagrams of $\lambda(a_2,\ldots,a_n)$.
\end{lemma}

\begin{proof}
If $M$ is a two-diagonal Tesler matrix, then for $i\geqslant 2$ \eqref{eq: tesler 2} simplifies to
\[
	m_{ii}+\cdots+m_{nn}+m_{i-1,i}=a_i+\cdots+a_n,
\]
while for $i=1$ we obtain
\[
    m_{11}+\cdots+m_{nn}=a_1+\cdots+a_n.
\]
This means that for  $i\geqslant 2$ the diagonal elements of $M$ define a subdiagram of $\lambda(a)$
\[
	m_{ii}+\cdots+m_{nn}\leqslant a_i+\cdots+a_n=\lambda_{i-1},
\]
while $m_{11}$ and all $m_{i-1,i}$ are uniquely determined by the diagonal. 
\end{proof}

 We define the functions $A(m)$ and $B(m)$ by the equations
\begin{equation*}
\begin{split}
	\sum_{m=0}^{\infty}A(m)z^m &= \frac{(1-z)(1-qtz)}{(1-qz)(1-tz)}=1-(1-q)(1-t)\frac{z}{(1-qz)(1-tz)}\\
	& \qquad \qquad \qquad \qquad \; = 1-(1-q)(1-t)\sum_{m=1}^{\infty}[m]_{q,t}z^m,\\
	\sum_{m=0}^{\infty}B(m)z^m &= \frac{1-z}{(1-qz)(1-tz)}=\sum_{m=0}^{\infty}([m+1]_{q,t}-[m]_{q,t})z^m.
\end{split}
\end{equation*}

\begin{theorem}
For all $a_i\geqslant 0$, we have
\begin{equation}
\label{eq:tesler big sum}
	F(a_2,\ldots,a_n)=\sum_{M} \prod_{i}B(m_{i,i+1})\prod_{j>i+1}A(m_{i,j}), 
\end{equation}
where the sum is over all Tesler matrices $M$ satisfying~\eqref{eq: tesler}.
\end{theorem}

\begin{proof}
The proof is very similar to \cite[Section 6.5]{GN.2015}, but we present it here for completeness. 
Since~\eqref{eq:tesler big sum} is an identity between rational functions in $q$ and $t$, without loss of generality 
we may assume that $q$ and $t$ are complex numbers very close to 1. Pick real numbers $1\ll r_1\ll \cdots \ll r_n$, 
and consider the torus
\[
	T=\{|z_1|=r_1,\ldots,|z_n|=r_n\}\subset \mathbb{C}^n.
\]
Given $a=(a_1,\ldots,a_n)$, consider the rational function 
\[
	\Phi_a(z_1,\ldots,z_n)=z_1^{a_1}\cdots z_n^{a_n}
	\prod_{i=1}^{n}\frac{1}{(1-z_i^{-1})}\prod_{i=2}^{n}\frac{1}{(1-qtz_{i-1}/z_i)}
	\prod_{i<j}\omega(z_i/z_j).
\]
We would like to prove that the integral 
\[
	I(a_1,\ldots,a_n)=\int_T \Phi_a(z_1,\ldots,z_n) \frac{dz_1}{2\pi iz_1}\cdots \frac{dz_n}{2\pi iz_n}
\]
equals both the left and the right hand side of~\eqref{eq:tesler big sum}.
First, we can write it as an iterated integral
\[
	I(a_1,\ldots,a_n)=\int_{|z_n|=r_n}\cdots \int_{|z_1|=r_1} 
	\Phi_a(z_1,\ldots,z_n) \frac{dz_1}{2\pi iz_1}\cdots \frac{dz_n}{2\pi iz_n}.
\]
Given $z_2,\ldots,z_n$, the possible poles of $\Phi_a(z_1,\ldots,z_n)$ in $z_1$ are at $z_1=1$, $z_1=z_k/q$ and 
$z_1=z_k/t$. By our choice of $r_i$, we observe that $z_1=1$ is the only pole inside the circle $|z_1|=r_1$, so the integral
\[
	R_1(z_2,\ldots,z_n)=\int_{|z_1|=r_1} \Phi_a(z_1,\ldots,z_n) \frac{dz_1}{2\pi iz_1}
\]
equals the residue at this pole, which is an explicit function in $z_2,\ldots,z_n$. Similarly, it is easy to see that for fixed 
$z_3,\ldots,z_n$ the only poles of $R_1(z_2,\ldots,z_n)$ are $z_2=q$ and $z_2=t$ (see Example \ref{example:residue 2}) 
and compute the integral
\[
	R_2(z_3,\ldots,z_n)=\int_{|z_2|=r_2}R_1(z_2,\ldots,z_n)\frac{dz_2}{2\pi iz_2}
\]
as a sum of residues at these poles. More generally, one can prove that for $a_i\geqslant 0$ the only poles that appear 
in the computation of $I(a_1,\ldots,a_n)$ are at points $(z_1,\ldots,z_n)$ corresponding 
to the $(q,t)$--contents of all standard tableaux, and \eqref{eq: def F} can be interpreted as a sum of residues 
at these poles. Therefore $I(a_1,\ldots,a_n)=F(a_2,\ldots,a_n)$.

On the other hand, we can change the order of integration and write
\[
	I(a_1,\ldots,a_n)=
	\int_{|z_1|=r_1}\cdots \int_{|z_n|=r_n} \Phi_a(z_1,\ldots,z_n) \frac{dz_n}{2\pi iz_n}\cdots \frac{dz_1}{2\pi iz_1}.
\]
For fixed $z_1,\ldots,z_{n-1}$ the possible poles are at $z_n=1,z_n=qz_k$ and $z_n=tz_k$ (note that the denominators
$(1-qt z_{i-1}/z_i)$ cancel out) which are all inside the circle $|z_n|=r_n$. Therefore the integral can be written as a residue
at infinity
\[
	\int_{|z_n|=r_n} \Phi_a(z_1,\ldots,z_n) \frac{dz_n}{2\pi iz_n}=-\Res_{z_n=\infty}\Phi_a(z_1,\ldots,z_n) \frac{dz_n}{2\pi iz_n},
\]
and similarly we have the iterated residue at infinity
\[
	I(a_1,\ldots,a_n)=(-1)^{n}\Res_{z_1=\infty}\cdots \Res_{z_n=\infty} \Phi_a(z_1,\ldots,z_n) 
	\frac{dz_n}{2\pi iz_n}\cdots \frac{dz_1}{2\pi iz_1}.
\]
To deal with these residues properly, we introduce new variables $u_i=z_i^{-1}$. Note that $z_i/z_j=u_j/u_i$. 
Hence $I(a_1,\ldots,a_n)$ equals
\[
	\Res_{u_1=0}\cdots \Res_{u_n=0} u_1^{-a_1}\cdots u_n^{-a_n}
	\prod_{i=1}^{n}\frac{1}{(1-u_i)}\prod_{i=2}^{n}\frac{1}{(1-qtu_{i}/u_{i-1})}
	\prod_{i<j}\omega(u_j/u_i)\frac{du_n}{2\pi iu_n}\cdots \frac{du_1}{2\pi iu_1}.
\]
which is precisely the coefficient of the rational function
\[
	\prod_{i=1}^{n}\frac{1}{(1-u_i)} \prod_{i=2}^{n}\frac{1}{(1-qtu_{i}/u_{i-1})}\prod_{i<j}\omega(u_j/u_i)
\]
at $u_1^{a_1}\cdots u_n^{a_n}$. On the other hand, we can expand the rational function as follows:
\begin{multline}
	\prod_{i=1}^{n}\frac{1}{(1-u_i)}\times \prod_{i=1}^{n-1}\frac{(1-u_{i+1}/u_{i})}{(1-qu_{i+1}/u_{i})(1-tu_{i+1}/u_{i})}\times
	\prod_{j>i+1}\omega(u_j/u_i)\\
	=\sum_{m_{ii}}u_i^{m_{ii}}\times \sum_{m_{i,i+1}}B(m_{i,i+1})\left(\frac{u_{i+1}}{u_{i}}\right)^{m_{i,i+1}}\times 
	\sum_{m_{i,j}}A(m_{i,j})\left(\frac{u_j}{u_{i}}\right)^{m_{i,j}}.
\label{eq: infinity expansion}
\end{multline}
The terms in the sum in~\eqref{eq: infinity expansion} are parameterized by the exponents $m_{ii},m_{i,i+1},m_{i,j}$ 
which can be combined in a single upper-triangular matrix $M=(m_{ij})$. Such a term contributes to 
$u_1^{a_1}\cdots u_n^{a_n}$ if
\[
	m_{ii}-\sum_{j>i}m_{ij}+\sum_{j<i}m_{ji}=a_i,
\]
which is precisely the Tesler matrix condition \eqref{eq: tesler}.
\end{proof}

\begin{example}
\label{example:residue 2}
For $n=2$, we have
\[
	\Phi_a(z_1,z_2)=
	\frac{z_1^{a_1}z_2^{a_2}(1-z_1/z_2)}{(1-z_1^{-1})(1-z_2^{-1})(1-qz_1/z_2)(1-tz_1/z_2)}.
\]
For fixed $z_2$, the poles are at $z_1=1,z_1=z_2/q$ and $z_1=z_2/t$, and only the first one is inside the circle $|z_1|=r_1$. 
Therefore
\[
	R_1(z_2)=\int_{|z_1|=r_1}\Phi_a(z_1,z_2)\frac{dz_1}{2\pi iz_1}
	=\frac{z_2^{a_2}(1-1/z_2)}{(1-z_2^{-1})(1-q/z_2)(1-t/z_2)}=\frac{z_2^{a_2}}{(1-q/z_2)(1-t/z_2)}.
\]
At the first step we compute the residue at $z_1=1$, and at the second we cancel the factors $(1-z_2^{-1})$. Now 
$R_1(z_2)$ has poles at $z_2=q$ and $z_2=t$, and the residues of $R_2(z_2)\frac{dz_2}{2\pi iz_2}$ are equal to 
$\frac{q^{a_2}}{1-t/q}$ and $\frac{t^{a_2}}{1-q/t}$, respectively.  

To compute the residue at infinity, we write $u_i=z_i^{-1}$ and
\[
	I(a_1,a_2)=\Res_{u_1=0}\Res_{u_2=0}\frac{u_1^{-a_1}u_2^{-a_2}(1-u_2/u_1)}
	{(1-u_1)(1-u_2)(1-qu_2/u_1)(1-tu_2/z_1)}\frac{du_2}{2\pi iu_2}\frac{du_1}{2\pi i u_1}.
\]
Now we expand
\begin{equation*}
\begin{split}
	&\frac{1}{1-u_1}=\sum_{m_{11}\geqslant 0}u_1^{m_{11}},\ 
	\frac{1}{1-u_2}=\sum_{m_{22}\geqslant 0}u_2^{m_{22}},\\
	&\frac{1-u_2/u_1}{(1-qu_2/u_1)(1-tu_2/u_1)}=\sum_{m_{12}\geqslant 0}B(m_{12})(u_2/u_1)^{m_{12}}.
\end{split}
\end{equation*}
By multiplying these three series and picking up the coefficient at $u_1^{a_1}u_2^{a_2}$ we get $m_{11}-m_{12}=a_1$, 
$m_{22}+m_{12}=a_2$, so $m_{11}$ and $m_{22}$ are determined by $m_{12}\leqslant a_2$ and
\[
	I(a_1,a_2)=\sum_{m_{12}\leqslant a_2}B(m_{12})=[a_2+1]_{q,t}.
\]
\end{example}

\begin{corollary}
For all $a_i\geqslant 0$ the function $F(a_2,\ldots,a_r)$ is a polynomial in $q$ and $t$. 
\end{corollary}

\begin{proof}
Indeed, by Lemma \ref{lem:tesler finite} there are finitely many terms in the sum~\eqref{eq:tesler big sum}, and for all 
$m\geqslant 0$ both $A(m)$ and $B(m)$ are polynomials in $q$ and $t$.
\end{proof}

\begin{corollary}
The specialization of $F(a_2,\ldots,a_r)$ at $t=1$ agrees with the sum
\[
	\sum_{\mu\subseteq \lambda(a)}q^{|\lambda(a)|-|\mu|},
\]
where $a=(a_2,\ldots,a_n)$.
\end{corollary}

\begin{proof}
It is clear that at $t=1$ the coefficients $A(m)$ and $B(m)$ specialize as follows:
\[
	A(m)\Big|_{t=1}=0\ \text{for}\ m>0, \quad A(0)\Big|_{t=1}=1, \quad  B(m)\Big|_{t=1}=q^{m}.
\]
Therefore at $t=1$ the sum \eqref{eq:tesler big sum} specializes to the sum over two-diagonal Tesler matrices which by 
Lemma~\ref{lem:subdiagrams} correspond to subdiagrams $\mu\subseteq \lambda(a)$. The weight of such a 
two-diagonal Tesler matrix specializes to $\prod_i q^{m_{i,i+1}}=q^{|\lambda(a)|-|\mu|}$.
\end{proof}

\begin{corollary}
\label{cor:set 0}
For $a_i\geqslant 0$ and $a_n=0$ we have
\[
	F(a_2,\ldots,a_{n-1},0)=F(a_2,\ldots,a_{n-1}).
\]
\end{corollary}

\begin{proof}
The last equation in \eqref{eq: tesler} reads as
\[
	m_{nn}+\sum_{j<n}m_{jn}=a_{n}.
\]
Hence if $a_n=0$, we obtain $m_{jn}=0$ for all $j$. Therefore a Tesler matrix with parameters $(a_1,a_2,\ldots,a_{n-1},0)$ is 
just an $(n-1)\times (n-1)$ Tesler matrix with row parameters $(a_1,a_2,\ldots,a_{n-1})$ completed with a column of 
zeroes. Since $A(0)=B(0)=1$, the weight of a Tesler matrix in \eqref{eq:tesler big sum} does not change after 
adding this column.
\end{proof}

%%%%%%%%%%%%%%%%%%%%%%%%%%%%%%%%%%%%%%%%%%%%%%%%%%%%%%%%%
\subsection{Separating the sum}

It is useful to separate the sum \eqref{eq: def F} into two pieces. Clearly, for any tableau $T$ with at least two boxes
either $z_2=q$ or $z_2=t$. Let us call a standard tableau $T$ \defn{head-like} if $z_2=q$. Given such a tableau, we define 
reduced weight $\twt(T)=(1-t/q)\wt(T)$ and 
$$
	H(a_2,\ldots,a_n;q,t)=\sum_{z_2(T)=q}z_2^{a_2}\cdots z_n^{a_n}\cdot \twt(T).
$$
Similarly to the proof of Proposition \ref{prop: F polynomial} one can prove that $H(a_2,\ldots,a_n)$ is a polynomial 
in $q$ and $t$ with integer coefficients.

\begin{remark}
\label{rem: a_2 monomial in H}
The polynomial $H(a_2,\ldots,a_n)$ depends on $a_2$ only by an overall factor of $q^{a_2}$:
$$
	H(a_2,\ldots,a_n;q,t)=q^{a_2}\left(\sum_{z_2(T)=q}z_3^{a_3}\cdots z_n^{a_n}\cdot \twt(T)\right).
$$
\end{remark}

\begin{remark}
In the geometric setup of Section \ref{subsec.fhilb}
the series $H(a_2,\ldots,a_n)$ computes the equivariant character of the pushforward 
$\pi_*(\CL_2^{a_2}\cdots \cdots \CL_n^{a_n})$ at one of the fixed points on $\FHilb^2(\CC^2,0)=\mathbb{P}^1$. 
Here $\pi \colon \FHilb^n(\CC^2,0)\to \FHilb^2(\CC^2,0)$ is the natural projection.
\end{remark}

The following  is clear from the definition:
\begin{equation}
\label{H to F}
	F(a_2,\ldots,a_n)=\frac{1}{1-t/q}H(a_2,\ldots,a_n;q,t)+\frac{1}{1-q/t}H(a_2,\ldots,a_n;t,q).
\end{equation}
Therefore any linear relation on $H(a)$ implies a linear relation for $F(a)$.

\begin{lemma}
\label{lem:positivity from H}
Assume that $H(a_2,\ldots,a_n)$ is a polynomial in $q$ and $t$ with nonnegative coefficients, where all monomials $q^it^j$
satisfy $i\geqslant j$. Then $F(a_2,\ldots,a_n)$ is a polynomial in $q$ and $t$ with nonnegative coefficients.
\end{lemma}

\begin{proof}
By linearity of~\eqref{H to F} it suffices to prove the statement for a single monomial $q^it^j$ with $i\geqslant j$. In this case
\begin{multline*}
	\frac{q^it^j}{1-t/q}+\frac{q^jt^i}{1-q/t}=q^jt^j\frac{q^{i-j+1}-t^{i-j+1}}{q-t}=q^jt^j(q^{i-j}+\cdots+t^{i-j})\\
	=q^it^j+q^{i-1}t^{j+1}+\cdots+q^{j+1}t^{i-1}+q^{j}t^{i}.
\end{multline*}
\end{proof}

\begin{corollary}
Assume that the polynomial $H(a_2,a_3,\ldots,a_n)$ has nonnegative coefficients. Then for all sufficiently large $N$ the polynomial $F(N,a_3,\ldots,a_n)$ has nonnegative coefficients.
\end{corollary}
\begin{proof}
Indeed, by Remark \ref{rem: a_2 monomial in H} we have
$$
H(N,a_3,\ldots,a_n)=q^{N-a_2}H(a_2,\ldots,a_n)
$$
and for sufficiently large $N$ all terms in it satisfy the condition in Lemma \ref{lem:positivity from H}.
\end{proof}

As we will see below, writing the formulas for $H(a;q,t)$ is much more efficient than the ones for $F(a;q,t)$, and the sums 
contain half as many terms.

\begin{example}
\label{example.n=2}
Consider the case $n=2$. There is only one tableau with $z_2(T)=q$, and $z(T)=(1,q)$. A direct computation shows 
that $\wt(1,q)=\frac{1}{1-t/q}$, so $\twt(1,q)=1$. Therefore $H(a)=q^{a}$. By the proof of 
Lemma~\ref{lem:positivity from H}, this confirms Example~\ref{example.2 and 3}.
\end{example}

\begin{example}
\label{example.n=3}
Consider the case $n=3$. There are two tableaux with $z_2=q$ and
$$
	\wt(1,q,q^2)=\frac{1}{(1-t/q)(1-t/q^2)}, \qquad \wt(1,q,t)=\frac{1}{(1-t/q)(1-q^2/t)}
$$
while 
$$
	\twt(1,q,q^2)=\frac{1}{(1-t/q^2)}, \qquad \twt(1,q,t)=\frac{1}{(1-q^2/t)}.
$$
We obtain
\begin{equation}
\label{H for n=3}
	H(a,b)=\frac{q^{a+2b}}{(1-t/q^2)}+\frac{q^{a}t^{b}}{(1-q^2/t)}=q^a(q^{2b}+q^{2b-2}t+\cdots+t^b)
	=q^a \sum_{i=0}^b (q^2)^{b-i}t^i.
\end{equation}
Note that~\eqref{H for n=3} holds for any integer $a$ and $b\geqslant -1$. Furthermore, $H(a,-1)=0$ for all integers $a$.
For $a\geqslant b\geqslant 0$ the conditions of Lemma~\ref{lem:positivity from H} are satisfied, and $F(a,b)$ has 
nonnegative coefficients. Using~\eqref{H to F}, one can confirm the explicit expression in Example~\ref{example.2 and 3}
(see also Lemma~\ref{F(a,b)}).
\end{example}

%%%%%%%%%%%%%%%%%%%%%%%%%%%%%%%%%%%%%%%%%%%%%%%%%%%%%%%%%
\subsection{Recursion for $n=4$}
\label{subsec:recursion n=4}

The situation for $n=4$ is more interesting. We record here the reduced weights $\twt(T)$ for all five head-like tableaux:
\begin{align*}
	\twt(1,q,q^2,q^3)& =\frac{1}{(1-t/q^2)(1-t/q^3)}, & \twt(1,q,q^2,t)&=\frac{1}{(1-t/q^2)(1-q^3/t)},\\
	\twt(1,q,t,q^2)&=\frac{(1-t)}{(1-t^2/q^2)(1-q^2/t)(1-t/q)}, & \twt(1,q,t,t^2)&=\frac{1}{(1-q^2/t^2)(1-q/t)},\\
	\twt(1,q,t,qt)& =\frac{1-q}{(1-q^2/t)(1-q/t)(1-t/q)}. &&
\end{align*}

\begin{lemma}
The polynomials $H(a,b,c)$ satisfy the following recursion
$$
	H(a,b,c)=H(a+1,b+1,c-1)+(qt)^cH(a+c,b-c)+\sum_{i=0}^{c-1}(qt)^{b+2c-2i}H(a-b-2c+4i).
$$
\end{lemma}

\begin{proof}
Let us compute the contribution of all  tableaux to $H(a,b,c)-H(a+1,b+1,c-1)$. Let $\ell(T;a,b,c)=z_2^az_3^bz_4^c.$
Then
\begin{equation*}
\begin{split}
	\ell(1,q,q^2,q^3;a,b,c)&=q^{a+2b+3c}=\ell(1,q,q^2,q^3;a+1,b+1,c-1),\\
	\ell(1,q,q^2,t;a,b,c)&=q^{a+2b}t^{c}, \qquad \ell(1,q,q^2,t;a+1,b+1,c-1)=q^{a+2b+3}t^{c-1},\\
	\ell(1,q,t,q^2;a,b,c)&=q^{a+2c}t^{b}, \qquad \ell(1,q,t,q^2;a+1,b+1,c-1)=q^{a+2c-1}t^{b+1},\\
	\ell(1,q,t,t^2;a,b,c)&=q^{a}t^{b+2c}, \qquad \ell(1,q,t,t^2;a+1,b+1,c-1)=q^{a+1}t^{b+2c-1},\\
	\ell(1,q,t,qt;a,b,c)&=q^{a+c}t^{b+c}=\ell(1,q,t,qt;a+1,b+1,c-1).
\end{split}
\end{equation*}

Therefore the contributions of $(1,q,q^2,q^3)$ and $(1,q,t,qt)$ cancel, and
\begin{multline*}
	H(a,b,c)-H(a+1,b+1,c-1)\\
	=q^{a+2b}t^{c}(1-q^3/t)\twt(1,q,q^2,t)+q^{a+2c}t^{b}(1-t/q)\twt(1,q,t,q^2)+q^{a}t^{b+2c}(1-q/t)\twt(1,q,t,t^2)\\
	=\frac{q^{a+2b}t^{c}}{(1-t/q^2)}+\frac{q^{a+2c}t^{b}(1-t)}{(1-t^2/q^2)(1-q^2/t)}+\frac{q^{a}t^{b+2c}}{(1-q^2/t^2)}.
\end{multline*}
On the other hand, by \eqref{H for n=3} we obtain
$$
	(qt)^cH(a+c,b-c)=\frac{q^{a+2b}t^c}{(1-t/q^2)}+\frac{q^{a+2c}t^{b}}{(1-q^2/t)},
$$
so 
$$
	H(a,b,c)-H(a+1,b+1,c-1)-(qt)^cH(a+c,b-c)=-\frac{q^{a+2c}t^{b}}{(1-q^2/t^2)}+\frac{q^{a}t^{b+2c}}{(1-q^2/t^2)}.
$$
Comparing this with the last term in the recurrence, we find
\begin{multline*}
	\sum_{i=0}^{c-1}(qt)^{b+2c-2i}H(a-b-2c+4i)=\sum_{i=0}^{c-1}(qt)^{b+2c-2i}\cdot q^{a-b-2c+4i}\\
	=\sum_{i=0}^{c-1}q^{a+2i}t^{b+2c-2i}=q^{a}t^{b+2c}\frac{1-q^{2c}t^{-2c}}{1-q^2/t^2}
	=\frac{q^{a}t^{b+c}-q^{a+2c}t^b}{1-q^2/t^2}.
\end{multline*}
\end{proof}

\begin{corollary}
\label{corollary.recursion}
The polynomials $F(a,b,c)$ satisfy the recursion relation
$$
	F(a,b,c)=F(a+1,b+1,c-1)+(qt)^cF(a+c,b-c)+\sum_{i=0}^{c-1}(qt)^{b+2c-2i}F(a-b-2c+4i).
$$
\end{corollary}

Note that the entries $a-b-2c+4i$ in the recurrence of Corollary~\ref{corollary.recursion} can become negative.
However, the following symmetry relation holds.

\begin{proposition}
\label{proposition.negative}
We have for $a>0$
\[
	F(-a) = - \frac{1}{(qt)^{a-1}} \; F(a-2).
\]
\end{proposition}

\begin{proof}
By~\eqref{H to F} and Example~\ref{example.n=2}, we have
\[
	F(a) = \frac{1}{1-t/q} q^{a} + \frac{1}{1-q/t} t^{a} = q^a \frac{1-(t/q)^{a+1}}{1-t/q}.
\]
Hence
\begin{equation*}
	F(-a) = q^{-a} \frac{1-(t/q)^{-a+1}}{1-t/q} = -q^{-a}(t/q)^{-a+1} \frac{1-(t/q)^{a-1}}{1-t/q}
	= - \frac{1}{(qt)^{a-1}} F(a-2).
\end{equation*}
\end{proof}

Note that using Corollary~\ref{corollary.recursion} and Proposition~\ref{proposition.negative}, $F(a,b,c)$ for 
$a\geqslant b \geqslant c \geqslant 0$ can be reduced to $F(a,b)$ for $a\geqslant b \geqslant 0$ and
$F(a)$ for $a\geqslant 0$, which are given in Example~\ref{example.2 and 3}.

Now we compute $F(a,b)$ explicitly by separating the sum.
\begin{lemma}\label{F(a,b)}
For $b \geqslant -1$ and $a \geqslant b-1$ we have 
\begin{equation*} 
	F(a,b)= \sum_{i=0}^b \sum_{j=i}^{a+2b-2i}q^{j}t^{(a+2b-i)-j}.
 \end{equation*}
\end{lemma}

\begin{proof}
We may express $F(a,b)$ in terms of $H(a,b)$ by separating the sum as above.  Using the expression for $H(a,b)$ 
from Example~\ref{example.n=3} (note that this expression is valid for $b \geqslant -1$ and any value of $a$) this gives us
\begin{equation*}
\begin{split}
	F(a,b) &= \frac{1}{1-t/q}q^a\sum_{i=0}^b (q^2)^{b-i}t^i+\frac{1}{1-q/t}t^a\sum_{i=0}^b (t^2)^{b-i}q^i \\
	&=\frac{1}{q-t}\bigg(q^{a+1}\sum_{i=0}^b (q^2)^{b-i}t^i-t^{a+1}\sum_{i=0}^b (t^2)^{b-i}q^i\bigg)\\
	&=\sum_{i=0}^b \frac{q^{a+2b+1-2i}t^i-t^{a+2b+1-2i} q^i}{q-t}=\sum_{i=0}^b \sum_{j=i}^{a+2b-2i}q^{j}t^{(a+2b-i)-j},
\end{split}
\end{equation*}
where the last step is legal because we are assuming $a \geqslant b-1$.
\end{proof}

\begin{lemma}
\label{lemma.-1}
We have $F(-1)=F(a,-1)=0$ for $a\geqslant -2$ and $F(a,b,-1)=0$ for $a,b\geqslant 1$.
\end{lemma}

\begin{proof}
Since $H(-1;q,t)=q^{-1}$, Equation~\eqref{H to F} implies $F(-1)=0$.  On the other hand, Lemma~\ref{F(a,b)} immediately 
implies $F(a,-1)=0$. Finally, by Corollary \ref{corollary.recursion} we have \begin{equation*}
F(a-1,b-1,0)=F(a-1+1,b-1+1,0-1)+(qt)^0 F(a-1+0,b-1-0).
\end{equation*}
But by Corollary \ref{cor:set 0} we have $F(a-1,b-1,0)=F(a-1,b-1)$ and the lemma follows.
\end{proof}

The recursion of Corollary~\ref{corollary.recursion} implies the following ``two-step'' recursion. It has the 
advantage that it does not contain any negative arguments, which will be advantageous for the combinatorial
analysis of Section~\ref{section.proof}.

\begin{lemma}
\label{lemma.two step recurrence}
For $a\geqslant b-1$, $a,b \geqslant c-1\geqslant 0$, we have
\begin{multline*}
	F(a,b,c) = F(a+2,b+2,c-2) + (qt)^c \; F(a+c,b-c) + (qt)^{c-1} \; F(a+c,b-c+2)\\
	+\sum_{j=2}^{\min(a-b,2c)} (qt)^{b+j} \; F(a-b+2c-2j)
	-\sum_{j=a-b+1}^{1} (qt)^{b+j} \; F(a-b+2c-2j).
\end{multline*}
\end{lemma}

\begin{remark}
If $a>b$ then the last sum is empty. If $a=b$ or $a=b-1$ then the next to last sum is empty, and the last sum contains one or two terms, respectively.
\end{remark}

\begin{proof}
Using the recurrence in Corollary~\ref{corollary.recursion} and then the same recurrence 
again on the term $F(a+1,b+1,c-1)$, we obtain
\begin{multline*}
	F(a,b,c) = F(a+2,b+2,c-2) + (qt)^c \; F(a+c,b-c) + (qt)^{c-1} \; F(a+c,b-c+2)\\
	+ \sum_{i=0}^{c-1} (qt)^{b+2c-2i} F(a-b-2c+4i) + \sum_{i=0}^{c-2} (qt)^{b+2c-1-2i} F(a-b-2c+2+4i).
\end{multline*}
The first three terms are the same as in the statement of the lemma. The last two sums can be combined to
\begin{equation*}
	\sum_{j=0}^{2c-2} (qt)^{b+2c-j} F(a-b-2c+2j),
\end{equation*}
or, reversing the order of the sum: 
\begin{equation}
\label{equation.combined sum}
	\sum_{j=2}^{2c} (qt)^{b+j} F(a-b+2c-2j).
\end{equation}
If $2c \leqslant a-b$ the corollary is proved. Otherwise we may break expression~\eqref{equation.combined sum}
above into two pieces to obtain
\begin{eqnarray*}
	\sum_{j=2}^{a-b} (qt)^{b+j} F(a-b+2c-2j) + 	\sum_{j=\max(a-b+1,2)}^{2c} (qt)^{b+j} F(a-b+2c-2j),
	\end{eqnarray*}
or equivalently 
\begin{multline*}
		\sum_{j=2}^{a-b} (qt)^{b+j} F(a-b+2c-2j) +\sum_{j=a-b+1}^{2c} (qt)^{b+j}F(a-b+2c-2j) \\
		-\sum_{j=a-b+1}^{1} (qt)^{b+j}F(a-b+2c-2j).
\end{multline*}
Therefore, if we show that the middle sum above is $0$ the corollary is proved.  However, setting $K=-a+b-2c$ we have
\[
	\sum_{j=a-b+1}^{2c} (qt)^{b+j}F(a-b+2c-2j)
	=\sum_{-K  \leqslant r \leqslant K-2} (qt)^{b-(K+r)/2}F(r),
\]
where the sum is over only those $r$ such that $2| (K+r)$.  Since $F(-1)=0$ this can be split into
\[
	\sum_{2  \leqslant r \leqslant K} (qt)^{b-(K-r)/2}F(-r)
	+\sum_{0  \leqslant s \leqslant K-2} (qt)^{b-(K+s)/2}F(s),
\]
where again the sum is only over $r$ with $2| (K+r)$.  However, applying  Proposition~\ref{proposition.negative} 
term-wise to the left sum gives precisely the opposite of the right sum.
\end{proof}

%%%%%%%%%%%%%%%%%%%%%%%%%%%%%%%%%%%%%%%%%%%%%%%%%%%%%%%%%
\section{Combinatorial expressions}
\label{section.combinatorial expressions}
%%%%%%%%%%%%%%%%%%%%%%%%%%%%%%%%%%%%%%%%%%%%%%%%%%%%%%%%%

In this section, we present a combinatorial formula for $F(a,b,c)$ when
$a+1 \geqslant b$, $a+1,b+1 \geqslant c\geqslant 0$. 

%%%%%%%%%%%%%%%%%%%%%%%%%%%%%%%%%%%%%%%%%%%%%%%%%%%%%%%%%
\subsection{Symmetric chain expression}
\label{section.sc}

Recall that $\lambda(a,b,c) = (a+b+c,b+c,c)$. We set $A= |\lambda(a,b,c)| = a + 2b +3c$, $\epsilon_{ij} = \max(0,i+j-b-c)$,
and $m_{cj} = c-j \pmod{2}$ for convenience. Define the \defn{symmetric chain} for $k\leqslant \ell$ as
\[
	[k,\ell]_{q,t} = q^\ell t^k + q^{\ell-1} t^{k+1} + \cdots + q^{k+1} t^{\ell-1} + q^k t^\ell.
\] 
We may write $F(a,b,c)$ as a sum of symmetric chains.

\begin{theorem}
\label{theorem.symmetric chains}
For nonnegative integers $a,b,c$ and $a+1 \geqslant b$, $a+1,b+1 \geqslant c$, we have
\[
	F(a,b,c) = \sum_{(i,j) \in \widetilde{Q}} [i+ \epsilon_{ij}, A-2i-j]_{q,t},
\]
where $\widetilde{Q} = \{(i,j) \mid 0\leqslant j \leqslant c, \; j\leqslant i \leqslant b+c, \; 2i+2j \leqslant a+b+2c-m_{cj}\}$.
\end{theorem}

The proof of Theorem~\ref{theorem.symmetric chains} is given in Section~\ref{section.proof}, see in
particular Corollary~\ref{corollary.F=Fcomb}. For the various conditions appearing in $\widetilde{Q}$, see 
the conditions for quasiheads in Table~\ref{table.indexing sets}. Note that Theorem~\ref{theorem.symmetric chains}
immediately implies that the right hand side is symmetric in $q$ and $t$.

\begin{remark}
Note that the conditions on $i$ and $j$ imply that $i+\epsilon_{ij} \leqslant A-2i-j$. Namely, since $i \leqslant b+c$,
we have $i+\epsilon_{ij} \leqslant \max(b+c,i+j)$. Furthermore, since $2i+2j \leqslant a+b+2c$, we have
$A-2i-j = A-2i-2j+j \geqslant b+c+j$ which in turn is greater or equal to $\max(b+c,i+j)$ given that $j\geqslant 0$ and 
$i\leqslant b+c$.
\end{remark}

\begin{remark}
Note that the interval $[i+ \epsilon_{ij}, A-2i-j]$ of integers that appears in the 
symmetric chains in Theorem~\ref{theorem.symmetric chains} will be called the \defn{area range} in 
Section~\ref{section.proof} as it is the range of the area statistic for the given symmetric chain.
\end{remark}

\begin{remark}
Surprisingly, the identity $H(a,b,c) = \sum_{(i,j) \in \widetilde{Q}} q^{A-2i-j} t^{i+\epsilon_{ij}}$ does not hold in general,
as the right hand side satisfies slightly different recursion relation, see Lemma \ref{lem: recursion H comb}.
\end{remark}

%%%%%%%%%%%%%%%%%%%%%%%%%%%%%%%%%%%%%%%%%%%%%%%%%%%%%%%%%
\subsection{Combinatorial expression}
\label{section.combinatorics}

The symmetric chain expression of Theorem~\ref{theorem.symmetric chains} leads to a purely combinatorial
expression for $F(a,b,c)$ as a sum of all subpartitions of $\lambda(a,b,c)$ with two associated statistics.
The \defn{area statistic} for $\lambda \subseteq \lambda(a,b,c)$ is given by 
\[
	\area(\lambda) = |\lambda(a,b,c)| - |\lambda|.
\]

The second statistic requires a little more notation. We set $L=a+b+c$.
Furthermore, we name the following cases:
\begin{enumerate}
\item[\textbf{Case 1.}] $z \geqslant \min(b+c-x, \lceil \frac{y-a}{2} \rceil)$
\begin{enumerate}
\item[\textbf{(a)}] $x+y-z+2 \epsilon_{yz}<L$
\item[\textbf{(b)}] $x+y-z+2 \epsilon_{yz}\geqslant L$
\begin{enumerate}
\item[\textbf{(i)}] $y+z<b+c$
\item[\textbf{(ii)}] $y+z\geqslant b+c$
\end{enumerate}
\end{enumerate}
\item[\textbf{Case 2.}] $z < \min(b+c-x, \lceil \frac{y-a}{2} \rceil)$.
\end{enumerate}
With this, we are ready to define the \defn{$t$-statistic}, where $\lambda=(x,y,z)$ is a partition
with $x\geqslant y \geqslant x \geqslant 0$ and $x\leqslant a+b+c$, $y \leqslant b+c$, $z\leqslant c$
\begin{equation}
\label{equation.statistics}
\stat(\lambda) = \begin{cases}
	x+\max(0,\lceil \frac{y-a}{2} \rceil, y+z-b-c,\lceil \frac{2y+z-L}{2} \rceil) & \text{in Case 1(a),}\\
	-L+2x+y-z+\max(0,\lceil \frac{L+z-x-a}{2} \rceil) & \text{in Case 1(b)(i),}\\
	2x+3y+z-(a+3b+3c) &\\
	\quad+\max(0,\lceil \frac{2b+2c-x-y}{2}\rceil, a+2b+2c-x-2y) & \text{in Case 1(b)(ii),}\\
	y+z & \text{in Case 2.}
	\end{cases}
\end{equation}

Our main result is the following.
\begin{theorem}
\label{theorem.combinatorics}
Let $a,b,c$ be nonnegative integers with $a+1 \geqslant b$, $a+1,b+1 \geqslant c$. Then
\[
	F(a,b,c) = \sum_{\lambda \subseteq \lambda(a,b,c)} q^{\mathsf{area}(\lambda)} t^{\mathsf{stat}(\lambda)}.
\]
\end{theorem}
The proof of Theorem~\ref{theorem.combinatorics} is given in Section~\ref{section.proof comb}.

\begin{example}
\label{example.321}
Let us take $a=b=c=1$, so that $\lambda(1,1,1)=(3,2,1)$. The subpartitions $\lambda$ of $(3,2,1)$ together with their
monomial $q^{\area(\lambda)} t^{\stat(\lambda)}$ are listed in Table~\ref{table.example321}, organized in the chains associated to Theorem~\ref{theorem.symmetric chains}.

\begin{table}
\caption{Subpartitions of $(3,2,1)$ with their monomials $q^{\area(\lambda)} t^{\stat(\lambda)}$
\label{table.example321}}
\begin{equation*}
\begin{array}{c|ccccccc}
\text{chains} & \multicolumn{7}{c}{\text{subpartitions with statistics}}\\
\toprule
\left[0,6\right]_{q,t}& \emptyset & {\tableau[scY]{ \cr}} & {\tableau[scY]{ ,}} & {\tableau[scY]{ ,,}}
  & {\tableau[scY]{ \cr,,}} & {\tableau[scY]{ ,\cr,,}} & {\tableau[scY]{ \cr,\cr,,}}\\
  & q^6 & q^5t & q^4t^2 & q^3t^3 & q^2 t^4 & q t^5 & t^6\\
  \midrule
\left[1,4\right]_{q,t}& {\tableau[scY]{ \cr \cr}} & {\tableau[scY]{ \cr,}} & {\tableau[scY]{ ,\cr,}} & {\tableau[scY]{ \cr ,\cr,}} &&&\\
  & q^4t & q^3 t^2 & q^2 t^3 & q t^4 &&&\\
  \midrule
\left[1,3\right]_{q,t}& {\tableau[scY]{ \cr \cr \cr}} &  {\tableau[scY]{ \cr \cr ,}} & {\tableau[scY]{ \cr \cr ,,}} &&&&\\
  & q^3t & q^2 t^2 & q t^3\\
\bottomrule
\end{array}
\end{equation*}
\end{table}
\end{example}

\begin{remark}
Note that $\stat(\lambda)$ is in general different from $\mathsf{dinv}(\lambda)$ and $\mathsf{bounce}(\lambda)$.
As stated in~\cite[Exercise 3.19]{Haglund.2008}, $\mathsf{dinv}(\lambda)$ is the number of cells $x$ in $\lambda$
such that $\mathsf{leg}(x) \leqslant \mathsf{arm}(\lambda) \leqslant \mathsf{leg}(x) +1$. Here $\mathsf{leg}(x)$
is the number of cells in $\lambda$ above $x$ in the same column as $x$ and $\mathsf{arm}(x)$ is the number of
cells in $\lambda$ to the right of $x$ in the same row as $x$. Then $q^{\area(\lambda)} t^{\mathsf{dinv}(\lambda)}$
for the partitions in Table~\ref{table.example321} read row by row, top to bottom, left to right are
\[
	q^6, q^5 t, q^4 t^2, q^3 t^2, q^2 t^4, q t^5, t^6, q^4t, q^3 t^3, q^2 t^3, q t^3, q^3 t, q^2 t^2, q t^4,
\]
which differs from Table~\ref{table.example321}. Similarly, one may check that $\mathsf{bounce}(\lambda)$
is in general different from $\stat(\lambda)$.
\end{remark}

\begin{example}
\label{example.432}
Consider $(a,b,c)=(1,1,2)$, so that $\lambda(1,1,2)=(4,3,2)$. The subpartitions $\lambda$ of $(4,3,2)$ together with their
monomial $q^{\area(\lambda)} t^{\stat(\lambda)}$ are listed in Table~\ref{table.example432} organized in the chains 
associated to Theorem~\ref{theorem.symmetric chains}.

\begin{table}
\caption{Subpartitions of $(4,3,2)$ with their monomials $q^{\area(\lambda)} t^{\stat(\lambda)}$
\label{table.example432}}
\begin{equation*}
\begin{array}{c|cccccccccc}
\text{chains} & \multicolumn{10}{c}{\text{subpartitions with statistics}}\\
\toprule
\left[0,9\right]_{q,t}& \emptyset & \scalebox{0.6}{\tableau[scY]{ \cr}} & \scalebox{0.6}{\tableau[scY]{ &}} & 
 \scalebox{0.6}{\tableau[scY]{ &&}} &\scalebox{0.6}{\tableau[scY]{ &&&}} & \scalebox{0.6}{\tableau[scY]{\cr&&&}} 
 & \scalebox{0.6}{\tableau[scY]{&\cr&&&}} & \scalebox{0.6}{\tableau[scY]{&&\cr&&&}}
 & \scalebox{0.6}{\tableau[scY]{\cr&&\cr&&&}} & \scalebox{0.6}{\tableau[scY]{&\cr&&\cr&&&}}\\
  & q^9 & q^8t & q^7t^2 & q^6t^3 & q^5 t^4 & q^4 t^5 & q^3t^6 & q^2 t^7 & q t^8 & t^9\\
  \midrule
\left[1,7\right]_{q,t}& \scalebox{0.6}{\tableau[scY]{ \cr \cr}} & \scalebox{0.6}{\tableau[scY]{ \cr &}} & 
  \scalebox{0.6}{\tableau[scY]{ \cr &&}} & \scalebox{0.6}{\tableau[scY]{&\cr&&}} &
  \scalebox{0.6}{\tableau[scY]{&&\cr&&}} & \scalebox{0.6}{\tableau[scY]{\cr &&\cr&&}} &
  \scalebox{0.6}{\tableau[scY]{& \cr &&\cr&&}} &&&\\
  & q^7t & q^6 t^2 & q^5 t^3 & q^4 t^4 & q^3 t^5& q^2 t^6& q t^7 &&&\\
  \midrule
\left[1,6\right]_{q,t}& \scalebox{0.6}{\tableau[scY]{ \cr \cr \cr}} &  \scalebox{0.6}{\tableau[scY]{ \cr \cr &}} & 
  \scalebox{0.6}{\tableau[scY]{ \cr \cr &&}} & \scalebox{0.6}{\tableau[scY]{ \cr \cr &&&}} &
  \scalebox{0.6}{\tableau[scY]{ \cr & \cr &&&}} & \scalebox{0.6}{\tableau[scY]{ &\cr & \cr &&&}} &&&&\\
  & q^6 t & q^5 t^2 & q^4 t^3 & q^3 t^4 & q^2 t^5 & q t^6 &&&& \\
  \midrule
\left[2,5\right]_{q,t}& \scalebox{0.6}{\tableau[scY]{ &\cr&}} &  \scalebox{0.6}{\tableau[scY]{ \cr &\cr&}} &
   \scalebox{0.6}{\tableau[scY]{ \cr &\cr &&}} & \scalebox{0.6}{\tableau[scY]{ &\cr &\cr &&}} &&&&&&\\
   & q^5 t^2 & q^4 t^3 & q^3 t^4 & q^2 t^5 &&&&&&\\
   \midrule
\left[3,3\right]_{q,t}& \scalebox{0.6}{\tableau[scY]{ &\cr& \cr &}} &&&&&&&&&\\
   & q^3 t^3 &&&&&&&&\\ 
\bottomrule
\end{array}
\end{equation*}
\end{table}
\end{example}

\begin{remark}
As the parameter $a$ becomes larger with respect to $b$ and $c$, simplifications occur.  
\begin{itemize}
\item When $a\geqslant b+c-1$, the statistic in \eqref{equation.statistics} can be simplified by eliminating 
Case 2 and setting any expression that appears inside a ``$\lceil \cdot \rceil"$ to $0$. Moreover, in 
Table~\ref{table.indexing sets} the parameters $\delta_{ij}$ and $\delta^{EF}$ become uniformly $0$ and the 
condition \eqref{p4} becomes unnecessary.  
\item When $a \geqslant b+2c$, all the above simplifications hold.  Moreover, in Table~\ref{table.indexing sets}
the conditions~\eqref{t3}, \eqref{p3}, and~\eqref{q3} become unnecessary.
\end{itemize}

\end{remark}
%%%%%%%%%%%%%%%%%%%%%%%%%%%%%%%%%%%%%%%%%%%%%%%%%%%%%%%%%
\section{Partition chains and proofs}
\label{section.proof}
%%%%%%%%%%%%%%%%%%%%%%%%%%%%%%%%%%%%%%%%%%%%%%%%%%%%%%%%%

In this section, we assume that $a \geqslant b-1$, $a,b \geqslant c-1$. 
We provide four different indexing sets for symmetric chains that partition the set 
\[
	\Lambda := \{ \lambda \mid \text{$\lambda \subseteq \lambda(a,b,c)$ and $\lambda$ a partition}\}
\]
called tails, pseudoheads, heads, and quasiheads. The tails, pseudoheads, and quasiheads are defined as
\begin{equation}
\begin{aligned}[3]
	&\text{Set of \defn{tails}}  &&T := \{T^{EF} \mid \text{conditions \eqref{t1}-\eqref{t3} on $E,F$} \},\\
	&\text{Set of \defn{pseudoheads}}  &&P := \{ P_{ij} \mid \text{conditions \eqref{p1}-\eqref{p4} on $i,j$} \},\\
	&\text{Set of \defn{quasiheads}}  &&Q := \{ Q_{st} \mid \text{conditions \eqref{q1}-\eqref{q3} on $s,t$} \},
\end{aligned}
\end{equation}
where $T^{EF}$, $P_{ij}$, and $Q_{st}$ are defined in Table~\ref{table.indexing sets} and 
for convenience $A=a+2b+3c$ and $L=a+b+c$ throughout this section.
In addition, we write $P=P^- \cup P^+$, where
\[
	P^- = \{ P_{ij} \in P \mid \delta_{ij} \leqslant \epsilon_{ij}\} \qquad \text{and} \qquad
	P^+ =  \{ P_{ij} \in P \mid \delta_{ij} > \epsilon_{ij}\}
\]
and $\epsilon_{ij}$ and $\delta_{ij}$ are also given in Table~\ref{table.indexing sets}.

Finally, we define the set of \defn{heads} $H=H^- \cup H^+$, where 
$H^-=P^-$ and 
\[
	H^+ = \{(k,\ell,0) \mid a < \ell \leqslant k <b+c\}.
\]
For a negative head, the area range is the same as its area range as a pseudohead.  For positive heads we set the area 
range to
\[
	 R_k^\ell=[\ell, A-k-\ell].
\]

\begin{example}
In terms of the indexing sets of Table~\ref{table.indexing sets}, the symmetric chains 
in Table~\ref{table.example321} of Example~\ref{example.321} from top to bottom correspond to the tails 
$T^{00}=(3,2,1)$, $T^{10}=(2,2,1)$, $T^{01}=(3,1,1)$, 
the pseudoheads (and heads) $P_{00}=(0,0,0)$, $P_{10}=(1,1,0)$, $P_{11}=(1,1,1)$, 
and the quasiheads $Q_{00}=(0,0,0)$, $Q_{10}=(1,1,0)$, $Q_{11}=(1,1,1)$, 
respectively. The tails are the largest partitions in the chain and the pseudoheads (heads, quasiheads) are
the smallest partitions in each chain.
\end{example}

\begin{example}
The symmetric chains in Table~\ref{table.example432} of Example~\ref{example.432} from top to bottom correspond to 
the tails $T^{00}=(4,3,2)$, $T^{10}=(3,3,2)$, $T^{01}=(4,2,2)$, $T^{11}=(3,2,2)$, $T^{21}=(2,2,2)$, 
the pseudoheads $P_{00}=(0,0,0)$, $P_{10}=(1,1,0)$, $P_{11}=(1,1,1)$, $P_{21}=(2,2,1)$, $P_{22}=(2,2,2)$, 
the heads $P_{00}=(0,0,0)$, $P_{10}=(1,1,0)$, $P_{11}=(1,1,1)$, $H_2^2=(2,2,0)$, $P_{22}=(2,2,2)$,
and the quasiheads $Q_{00}=(0,0,0)$, $Q_{10}=(1,1,0)$, $Q_{11}=(1,1,1)$, $Q_{20}=(2,2,0)$, $Q_{22}=(2,2,2)$,
respectively. The tails are the largest partitions in the chain and the heads are the smallest partitions in the chain.
For the chain $[2,5]_{q,t}$, the head and pseudohead are not the same.
\end{example}

The set of tails, pseudoheads, heads, and quasiheads are all in area preserving bijection. That is, if $X,Y$ 
are one of the sets tails, pseudoheads, heads, and quasiheads and the area ranges for $x\in X$ and $y\in Y$ are $R_x$ 
and $R_y$, respectively, then there is a bijection $\Phi \colon X \to Y$ such that $R_x=R_{\Phi(x)}$ for all $x\in X$
(see Sections~\ref{section.pseudoheads}, \ref{section.heads} and~\ref{section.quasi}).

In Section~\ref{section.chains}, we define chains (using the strings of Section~\ref{section.strings})
and prove in Theorem~\ref{theorem.chain decomp} that the chains partition $\Lambda$, the set of all
subpartitions of $\lambda(a,b,c)$. In Section~\ref{section.combinatorial recursion}, using the quasiheads, 
we show that the combinatorial symmetric chain function $G(a,b,c)$ satisfies the same recursions as $F(a,b,c)$, 
thereby proving Theorem~\ref{theorem.symmetric chains}. The proof of Theorem~\ref{theorem.combinatorics}
is given in Section~\ref{section.proof comb}.

%%%%%%
\begin{table}[!h]
\caption{Various indexing sets for chains \label{table.indexing sets}}
\begin{tabularx}{\linewidth}{l|p{10cm}}
\toprule
\defn{Tails} & $\qquad T^{EF} = (a+b+c-E,b+c-F,c)$\\
\hline
Conditions & \vspace{-0.5cm}{
\begin{subequations}
\begin{align}
	&\hspace{-6.8cm}0 \leqslant F \leqslant c-\epsilon_{EF} \hspace{-1cm}\label{t1}\\
	&\hspace{-6.8cm}2\epsilon^{EF} \leqslant E \leqslant F+a \hspace{-1cm}\label{t2}\\
	&\hspace{-6.8cm}4E+5F-3\epsilon^{EF} \leqslant a+3b+3c \hspace{-1cm}\label{t3}
\end{align}
\end{subequations}}\\[-5mm]
\hline	        
Area range & $\qquad R^{EF}= [E+F ,A-2E-3F+ \max(\epsilon^{EF}, \delta^{EF})]$\\
\hline
Notation & $\qquad \epsilon^{EF} = \max(0,E+2F-b-c)$ and $\delta^{EF} = \lceil\frac{E+F-a}{2}\rceil $\\
\bottomrule
\end{tabularx}

\vspace{4mm}

\begin{tabularx}{\linewidth}{l|p{10cm}}
\toprule
\defn{Pseudoheads} & $\qquad P_{ij} = (i,i,j)$\\
\hline
Conditions & \vspace{-0.5cm}{
\begin{subequations}
\begin{align}
	         & \hspace{-6.8cm}0 \leqslant j \leqslant c \hspace{0.7cm}\label{p1}\\
	         & \hspace{-6.8cm}j \leqslant  i \leqslant b+c \hspace{0.7cm}\label{p2}\\
	         & \hspace{-6.8cm}4i+j \leqslant a+3b+3c \hspace{0.7cm} \label{p3}\\
         	& \hspace{-6.8cm}i-2j \leqslant a \hspace{0.7cm} \label{p4}
\end{align}
\end{subequations}}\\[-5mm]
\hline
Area range & $\qquad R_{ij}= [i+ \epsilon_{ij}, A-2i-j+\max(0,\delta_{ij}-\epsilon_{ij})]$\\
\hline
Notation & {
\begin{flalign} \hspace{-6.8cm}\epsilon_{ij} = \max(0,i+j-(b+c)), \qquad
	\delta_{ij} = \lceil\frac{i+\epsilon_{ij}-a}{2}\rceil \label{equation.eps ij} \hspace{-4.7cm}\end{flalign}}\\[-5mm]
\bottomrule
\end{tabularx}

\vspace{4mm}

\begin{tabularx}{\linewidth}{l|p{10cm}}
\toprule
\defn{Quasiheads} & $\qquad Q_{st} = (s,s,t)$\\
\hline
Conditions & \vspace{-0.5cm}{
\begin{subequations}
\begin{align}
	          & \hspace{-6.8cm}0 \leqslant t \leqslant c \label{q1}\\
	         & \hspace{-6.8cm}t \leqslant s \leqslant b+c\label{q2}\\
	         & \hspace{-6.8cm}2s+2t \leqslant a+b+2c-m_{ct}\label{q3}
\end{align}
\end{subequations}}\\[-5mm]
\hline	        
Area range & $\qquad R\mms{st}=[s+ \epsilon_{st}, A-2s-t]$\\
\hline
Notation & $\qquad \epsilon_{st} = \max(0,s+t-(b+c))$ and $m_{ct} = (c-t) \pmod{2}$\\
\bottomrule
\end{tabularx}
\end{table}

%%%%%%%%%%%%%%%%%%%%%%%%%%%%%%%%%%%%%%%%%%%%%%%%%%%%%%%%%
\subsection{Area preserving bijection between tails and pseudoheads}
\label{section.pseudoheads}
We now construct an area preserving bijection between tails and pseudoheads.
\begin{lemma}
Define maps $\Psi$ and $\Psi^{-1}$ by
\begin{align*}
	 \Psi(E,F)&=(E+F-\epsilon^{EF},F+\epsilon^{EF}),\\
 	\Psi^{-1}(i,j)&=(i-j+2\epsilon_{ij}, j-\epsilon_{ij}).
\end{align*}
Then $\Psi$ induces  a bijection from $T$ to $P$ via the rule that if $\Psi(E,F)=(i,j)$ then 
\[
	(a+b+c-E,b+c-F,c) \mapsto (i,i,j).
\]
The inverse of this bijection is induced by $\Psi^{-1}$  via the rule that if $\Psi^{-1}(i,j)=(E,F)$ then
\[
	(i,i,j) \mapsto (a+b+c-E,b+c-F,c).
\]
Moreover, if $\Psi(E,F)=(i,j)$, then $R^{EF}=R_{ij}$.
\end{lemma}

\begin{proof}
First we show that $\Psi$ is a bijection on $\mathbb{Z}^2$.  Indeed, note that if either $\Psi(E,F)=(i,j)$ or $\Psi^{-1}(i,j)=(E,F)$ 
we have $\delta^{EF}=\delta_{ij}$ and $\epsilon^{EF}=\epsilon_{ij}$.  Hence, a simple computation shows that  
$\Psi \circ \Psi^{-1}$ and  $\Psi^{-1} \circ \Psi$ are the identity on $\mathbb{Z}^2$.  Moreover, it is apparent that $\Psi$ 
preserves the area range. It remains to show that $\Psi(T) \subseteq P$ and $\Psi^{-1}(P)\subseteq T$.

First let $T^{EF} \in T$ and suppose $\Psi(E,F)=(i,j)$.  We must show that the conditions in~\eqref{p1}-\eqref{p4} hold:
 \begin{itemize}
 \item Condition~\eqref{p1}: The condition $0 \leqslant j \leqslant c$ translates to $0 \leqslant F+ \epsilon^{EF} \leqslant c$ 
 which is immediate from \eqref{t1}.
 \item Condition \eqref{p2}: The condition $j \leqslant i \leqslant b+c$ translates to $F+ \epsilon^{EF} \leqslant E+F -\epsilon^{EF}
  \leqslant b+c$. The left hand side follows from the left hand side of \eqref{t2}.  If $\epsilon^{EF}=0$, then we have 
  $E+2F \leqslant b+c$ so the right hand side follows.  
  Otherwise the right hand side reduces to $-F+b+c \leqslant b+c$ which follows from $F \geqslant 0$.
 \item Condition \eqref{p3}: The condition $4i+j \leqslant a+3b+3c$ translates to $4E+5F-3\epsilon^{EF} \leqslant a+3b+3c$
  which is \eqref{t3}.
 \item Condition \eqref{p4}: The condition $i-2j \leqslant a$ translates to $E-F-3\epsilon^{EF} \leqslant a$ which follows from 
 the right hand side of \eqref{t2}.
 \end{itemize}
 
Now let $P_{ij} \in P$ and suppose $\Psi^{-1}(i,j)=(E,F)$.  We must show that the conditions in \eqref{t1}-\eqref{t3} hold:
\begin{itemize}
\item Condition \eqref{t1}: The condition $0 \leqslant F \leqslant c-\epsilon^{EF}$ translates to 
$0 \leqslant j-\epsilon_{ij} \leqslant c-\epsilon^{EF}$. The left hand side follows from $j \geqslant 0$ unless $\epsilon_{ij}>0$, 
in which case it follows from $i \leqslant b+c$. The right hand side is equivalent to $j \leqslant c$ (since 
$\epsilon_{ij}=\epsilon^{EF}$).
\item Condition \eqref{t2}: The condition $2\epsilon^{EF} \leqslant E \leqslant F+a$ translates to 
$2\epsilon^{EF} \leqslant i-j+2\epsilon_{ij} \leqslant j-\epsilon_{ij}+a$.
The left hand side is equivalent to $j \leqslant i$ and the right hand side follows from \eqref{p4}.
\item Condition \eqref{t3}: The condition $4E+5F-3\epsilon^{EF} \leqslant a+3b+3c $ translates to 
$4i-4j+8\epsilon_{ij} +5j -5 \epsilon_{ij}-3\epsilon^{EF} \leqslant a+3b+3c $ which follows from \eqref{p3}.
\end{itemize}
\end{proof}

%%%%%%%%%%%%%%%%%%%%%%%%%%%%%%%%%%%%%%%%%%%%%%%%%%%%%%%%%
\subsection{Area preserving bijection between pseudoheads and heads}
\label{section.heads}

We now construct an area preserving bijection between pseudoheads and heads.
Set $\delta_k^\ell=\lceil\frac{\ell-a}{2}\rceil$ and $\epsilon_k^\ell=\max(k+\delta_k^\ell-b-c,0)$. 
\begin{lemma}
Define maps $\Theta$ and $\Theta^{-1}$ by
\begin{equation*}
\begin{split}
	 \Theta(i,j)&=(i+j-\delta_{ij},i+\epsilon_{ij}),\\
	 \Theta^{-1}(k,\ell)&=(\ell-\epsilon_k^\ell,k-\ell+\epsilon_k^{\ell}+\delta_k^\ell).
\end{split}
\end{equation*}
Then $\Theta$ induces a bijection from $P$ to $H$, which is the identity on $P^-$ and, if $\Theta(i,j)=(k,\ell)$ it acts as
$(i,i,j) \mapsto (k,\ell,0)$ on $P^+$.  The inverse of this map is the identity on $H^-$ and, if $\Theta^{-1}(k,\ell)=(i,j)$,
then $(k,\ell,0) \mapsto (i,i,j)$ on $H^+$.  Moreover if $\Theta(i,j)=(k,\ell)$ then $R_{ij}=R_k^\ell$.
\end{lemma}

\begin{proof}
First we show that $\Theta$ is a bijection on $\mathbb{Z}^2$.  Indeed, note that if either $\Theta(i,j)=(k,\ell)$ or 
$\Theta^{-1}(k,\ell)=(i,j)$ we have $\delta_k^\ell=\delta_{ij}$ and $\epsilon_k^\ell=\epsilon_{ij}$.  Hence, a simple 
computation shows that $\Theta \circ \Theta^{-1}$ and $\Theta^{-1} \circ \Theta$ are the identity on $\mathbb{Z}^2$.  
Moreover, it is apparent that $\Theta$ preserves the area range.

Now suppose $P_{ij} \in P^+$, and $\Theta(i,j)=(k,\ell)$. We wish to show that $(k,\ell,0) \in H^+$.  This means we must 
verify the inequalities $a< i+ \epsilon_{ij} \leqslant i+j-\delta_{ij} <b+c$. The first inequality is immediate because 
$\delta_{ij} > \epsilon_{ij}$ is equivalent to $i-\epsilon_{ij} > a$.  The second inequality is the same as 
$\delta_{ij}+\epsilon_{ij} \leqslant j$ which is equivalent to $i-2j \leqslant a- 3 \epsilon_{ij}$.  If $\epsilon_{ij}=0$,
this is the same as the pseudohead condition $i-2j \leqslant a$.  Otherwise, it is equivalent to the pseudohead condition 
$4i+j \leqslant a+3b+3c$.  Finally, the last inequality is just $i+j-(b+c)< \delta_{ij}$ which is immediate since the former is 
less than or equal to $\epsilon_{ij}$ which is by assumption less than $\delta_{ij}$.

Now suppose $H_k^\ell \in H^+$ and $\Theta^{-1}(k,\ell)=(i,j)$.  We need to show that $\delta_{ij}>\epsilon_{ij}$ as well 
as  the pseudohead conditions \eqref{p1}-\eqref{p4} for $i=\ell-\epsilon_k^\ell$ and $j=k-\ell+\epsilon_k^{\ell}+\delta_k^\ell$ 
for any $(k,\ell)$ such that $a < \ell \leqslant k <b+c$:
\begin{itemize}
\item $\delta_{ij}>\epsilon_{ij}$. We have $\delta_{ij}-\epsilon_{ij}=\delta_k^\ell-\epsilon_k^\ell=\min(-k+b+c,\delta_k^\ell)$.
But this is a positive number because $k<b+c$ and $\ell>a$.
\item Condition \eqref{p1}: The condition $0 \leqslant j \leqslant c$ translates to $0 \leqslant k-\ell +\epsilon_k^\ell
+ \delta_k^\ell \leqslant c$. The left side holds since all of $k-\ell,\epsilon_k^\ell,\delta_k^\ell$ are nonnegative. Now, if 
$\epsilon_k^\ell=0$ then $k+\delta_k^\ell \leqslant b+c \leqslant a+c+1 \leq \ell+c$ which implies the right hand side.  
On the other hand, if $\epsilon_k^\ell>0$ the inequality becomes $2k-\ell+2\delta_k^\ell \leqslant b+c$ which would 
certainly hold if $2k +2\frac{\ell-a}{2}-\ell=2k-a<b+2c$. But this is true since $k<b+c$ and $k \leqslant b+c-1 \leqslant a+c$.
\item Condition \eqref{p2}: The left hand side of the condition $j\leqslant i \leqslant b+c$ translates to 
$ k-\ell+\epsilon_k^{\ell}+\delta_k^\ell \leqslant \ell-\epsilon_k^\ell$, that is, $k+2 \epsilon_k^\ell \leqslant 2\ell-\delta_k^\ell$.  
If $\epsilon_k^\ell=0$ this says $k \leqslant \lfloor\frac{3\ell+a}{2}\rfloor$. But 
$k \leqslant b+c-1 \leqslant (a+1) +(a+1) -1 =2a+1$.  
On the other hand $\ell>a$ implies $\lfloor\frac{3\ell+a}{2}\rfloor \leqslant 2a+1$. If $\epsilon_k^\ell>0$ the left hand inequality 
reduces to $3k \leqslant 2\ell-3\delta_k^\ell +2b+2c=\lfloor \frac{\ell+a}{2} \rfloor +a +2b +2c$ which would certainly hold 
if $2k+k=3k \leqslant 2a+2b+2c$.  But $2k \leqslant b+c-2$ and $k \leqslant 2a+1$ so this holds (in fact strictly).  
Moreover, the righthand side easily holds as $\ell-\epsilon_k^\ell \leqslant \ell \leqslant k <b+c$.
\item Condition \eqref{p3}: The condition $4i+j \leqslant a+3b+3c$ translates to $3\ell-3\epsilon_k^\ell +k+ \delta_k^\ell
 \leqslant a+3b+3c$.  If $\epsilon_k^\ell=0$, we have $k+ \delta_k^\ell \leqslant b+c$. Hence it is enough to show that
 $3\ell \leqslant a+2b+2c$.  But this is also true since $k+ \delta_k^\ell \leqslant b+c$ is equivalent to $2k+\ell \leqslant
 a+2b+2c$ and $\ell \leqslant k$. On the other hand, if $\epsilon_k^\ell>0$ the inequality we need to show reduces to 
 $3\ell-2k-2\delta_k^\ell \leqslant a$. Since $\ell-k \leqslant 0$ it suffices to show that $\ell- 2\delta_k^\ell \leqslant a$. 
 But this is clear since $\ell- 2\delta_k^\ell \leqslant \ell-2\frac{\ell-a}{2} = a$. 
\item Condition \eqref{p4}: The condition $i-2j \leqslant a$ translates to $3\ell-2k-3\epsilon_k^\ell-3\delta_k^\ell \leqslant a$.
But this follows from $\ell-2\delta_k^\ell \leqslant a$ and $\ell-k \leqslant 0$.
\end{itemize}
This shows that $\Theta$ induces a bijection from $P^+$ to $H^+$. Extending this map to all of $P$ by declaring it to be 
the identity on $P^-$ is also a bijection as long as $H^- \cap H^+ = \emptyset$.  Indeed, only 
partitions of the form $(m,m,0)$ may lie in both $H^-$ and $H^+$.  Moreover, being in $H^-$ implies 
$\delta_{m0}-\epsilon_{m0} \leqslant 0$ which means $m \leqslant a$.  On the other hand, being in $H^+$ would 
require $a<m$.
\end{proof}

%%%%%%%%%%%%%%%%%%%%%%%%%%%%%%%%%%%%%%%%%%%%%%%%%%%%%%%%%
\subsection{Strings}
\label{section.strings}
We now consider the set of all partitions $\Lambda$ which fit inside the partition $\lambda(a,b,c)=(a+b+c,b+c,c)$.  We call such 
a partition $(x,y,z)$ \defn{positive} if $z<\min(b+c-x,\lceil \frac{y-a}{2}\rceil)$ and negative otherwise.
Write $\Lambda=\Lambda^- \cup \Lambda^+$.  

Let $P_{ij}=(i,i,j)$ be a pseudohead with $\Psi(E,F)=(i,j)$. Suppose that $T^{EF}=(p,q,c)$.  We define the string 
associated to $P_{ij}$ and $T^{EF}$ to be
\begin{equation}
\label{equation.SP union}
	S(P_{ij})=S(T^{EF})=\bigcup_{i\leqslant x < p} (x,i,j)  \bigcup_{i \leqslant y < q} (p,y,j) 
	\bigcup_{j \leqslant z \leqslant c} (p,q,z).
\end{equation}

\begin{lemma}
$(p,q,c)$ is a partition containing $(i,i,j)$ and is contained in $\lambda(a,b,c)$.  Thus every $S(P_{ij})$ is a 
nonempty set of partitions contained in $\lambda(a,b,c)$.  
\end{lemma}

\begin{proof}
It is clear that $p=L-E \leqslant L$ by the left side of~\eqref{t2}. Furthermore, $q=b+c-F \leqslant b+c$ by the left side 
of~\eqref{t1}. Obviously $c \leqslant c$. Hence $(p,q,c)$ is contained in $\lambda(a,b,c)$.

Now $p-q=L-E-(b+c-F)=a-E+F \geqslant 0$ by the right side of \eqref{t2}.  Furthermore, $q=b+c-F \geqslant c$ follows from 
$F \leqslant c$ (which comes from the right side of \eqref{t1}) unless $b<c$.  If $b<c$, we must have $b=c-1$ so we just 
need to show $q=(c-1)+c-F \geqslant c$ 
or equivalently $F \leqslant c-1$.  Indeed, if $F=c$ then $\epsilon^{EF}=\max(2c+E-(2c-1),0)=\max(E+1,0)>0$ 
by the left-hand side of~\eqref{t2}. Thus the right-hand side of \eqref{t1} implies $F \leqslant c-1$ contradicting the 
assumption $F=c$. This shows that $(p,q,c)$ is indeed a partition.

Finally it is obvious that $j \leqslant c$ and since we already showed that $p \geqslant q$ all that remains to show is 
$q \geqslant i$.  But this says $b+c-F \geqslant i$ or $b+c-(j-\epsilon_{ij}) \geqslant i$ which is equivalent to 
$i+j-(b+c) \leqslant \epsilon_{ij}$ which follows immediately from~\eqref{equation.eps ij}.
\end{proof}

\begin{theorem} \label{neg}
Let $\mu \in \Lambda^-$. Then there exists unique  $P_{ij} \in P$ such that $\mu \in S(P_{ij})$.  Conversely, if 
$\mu \in S(P_{ij})$ for some pseudohead $P_{ij}$, then $\mu \in \Lambda^-$.
\end{theorem}

\begin{proof}
Let $\mu=(x,y,z) \in \Lambda^-$. Let us set:
\[
	\mathcal{E}(y,z)=y-z+2\epsilon_{yz} \qquad \text{and} \qquad
	\mathcal{F}(y,z)=z-\epsilon_{yz}.
\]
We prove the first statement in three cases.
\begin{enumerate}
\item{First suppose $x+\mathcal{E}(y,z)<L$. Note that this corresponds to Case 1(a) in Section~\ref{section.combinatorics}.
To show uniqueness suppose $\mu \in S(P_{ij})$ with tail $T^{(L-p)(b+c-q)}$. 

If $\mu$ is from the second union in~\eqref{equation.SP union}, then $\mu=(p,y,j)$ for $i \leqslant y <q$. Since 
$\mathcal{E}(y,j) \geqslant \mathcal{E}(i,j)$ we have:
\begin{equation*}
	x+\mathcal{E}(y,z)=p+\mathcal{E}(y,j) \geqslant p+ \mathcal{E}(i,j)=  (L-\mathcal{E}(i,j))+\mathcal{E}(i,j)=L.
\end{equation*}

If $\mu$ is from the third union in~\eqref{equation.SP union}, then $\mu = (p,q,z)$ for $j \leqslant z \leqslant c$. 
Now $q=b+c-\mathcal{F}(i,j)$ implies $q+j \geqslant b+c$. From this, it follows that $\mathcal{E}(q,z) \geqslant 
\mathcal{E}(q,j)$.  Since $\mathcal{E}(q,j) \geqslant \mathcal{E}(i,j)$ as well we have:
\begin{equation*}
	x+\mathcal{E}(y,z)=p+\mathcal{E}(q,j)  \geqslant p+ \mathcal{E}(i,j)=  (L-\mathcal{E}(i,j))+\mathcal{E}(i,j)=L.
\end{equation*}
This means that $\mu$ can only come from the first union, so that we must have $i=y$ and $j=z$.  Hence $\mu$ can be 
in no other string than $S(P_{yz})$.

Now we show that $\mu \in S(P_{yz})$.  First we need to check $P_{yz}$ satisfies the pseudohead conditions:
\begin{itemize}
\item Condition \eqref{p1}: $0 \leqslant z \leqslant c$ is immediate.
\item Condition \eqref{p2}: $z\leqslant y \leqslant b+c$ is immediate.
\item Condition \eqref{p3}: $4y+z \leqslant a+3b+3c$.  If $\epsilon_{yz}=0$ then the original assumption becomes $x+y-z<L$ 
and we also have $y+z \leqslant b+c$.  Adding the first inequality to twice the second yields $x+3y+z < a+3b+3c$ and we 
are done since $y \leqslant x$.  If $\epsilon_{yz}>0$ then the original assumption reduces directly to $x+3y+z < a+3b+3c$ 
so we are done for the same reason.
\item Condition \eqref{p4}: $y-2z \leqslant a$. First suppose $\epsilon_{yz}=0$.  Now, since $\mu$ is a negative partition we 
either have $z \geqslant \lceil \frac{y-a}{2}\rceil$ which would mean $2z \geqslant y-a$ and we would be done, or else, 
$z \geqslant b+c-x$. In the second case: $\epsilon_{yz}=0$ along with the original assumption  imply $x+y-z<L$, and 
subtracting from this the inequality $x+z \geqslant b+c$ gives $y-2z<a$.  Finally, if $\epsilon_{yz}>0$ then we have 
$y+z>b+c$.  Subtracting three times this from $4y+z \leqslant a+3b+3c$ (which we have already verified) gives $y-2z<a$.
\end{itemize}
Now that $P_{yz}$ is in fact a pseudohead it is obvious that $\mu \in S(P_{yz})$
(in the first union) because $x<L-\mathcal{E}(y,z)$.}

\item{Now suppose $x+\mathcal{E}(y,z)\geqslant L$ and $y+z<b+c$. Note that this corresponds to 
Case 1(b)(i) in Section~\ref{section.combinatorics}. To show uniqueness suppose $\mu \in S(P_{ij})$ 
with tail $T^{(L-p)(b+c-q)}$. 

If $\mu$ is from the first union in~\eqref{equation.SP union}, then $\mu=(x,i,j)$ for $i \leqslant x <p$. But this means 
$x+\mathcal{E}(i,j)<L$, that is, $x+\mathcal{E}(y,z)<L$, contradicting our assumption.

If $\mu$ is from the third union in~\eqref{equation.SP union}, then $\mu = (p,q,z)$ for $j \leqslant z \leqslant c$.
Now $q=b+c-\mathcal{F}(i,j)$ where $\mathcal{F}(i,j)=j$ because $y+z < b+c$ means $i+j <b+c$.  Thus $q+j=b+c$ so 
$q+z \geqslant b+c$, that is, $y+z \geqslant b+c$, again contradicting our assumption.

This means that $\mu$ can only come from the second union in~\eqref{equation.SP union}.
In this case, $\mu$ is of the form $(p,y,j)$ for $i \leqslant y<q$.  
In particular, $x=p=L-\mathcal{E}(i,j)=L-\mathcal{E}(i,z)$.  But $i+z \leqslant y+z <b+c$ so $\epsilon_{iz}=0$ and this reduces to 
$x=L-i+z$.  Therefore, $i=L+z-x$, and we see $\mu$ can be in no other string than $S(P_{(L+z-x)z})$.

Now we show that $\mu \in S(P_{(L+z-x)z})$.  First we need to check that $P_{(L+z-x)z}$ satisfies the pseudohead conditions:
\begin{itemize}
\item Condition \eqref{p1}: $0 \leqslant z \leqslant c$ is immediate.
\item Condition \eqref{p2}: $z\leqslant L+z-x \leqslant b+c$.  The left hand side is immediate because $x \leqslant L$.
On the other hand the first original assumption implies $L+z-x \leqslant \mathcal{E}(y,z)+z$ and the second original 
assumption implies $\mathcal{E}(y,z)=y-z$.  Thus $L+z-x \leqslant y \leqslant b+c$.
\item Condition \eqref{p3}: $4(L+z-x)+z \leqslant a+3b+3c$.  Since $\mu$ is a negative partition we have 
$z \geqslant \min(b+c-x,\lceil \frac{y-a}{2}\rceil)$.  First suppose that $z \geqslant \lceil \frac{y-a}{2}\rceil$.  
This along with the fact that $(L+z-x)+z \leqslant y+z <b+c$ implies that: 
\begin{eqnarray*}
4(L+z-x)+z=(L+z-x)+3(L+z-x+z)-2z \\ 
\leqslant y+3(b+c)+(a-y) = a+3b+3c.
\end{eqnarray*}
Otherwise we must have $z< \lceil \frac{y-a}{2}\rceil$, but $z \geqslant b+c-x$. Now $\lceil \frac{y-a}{2}\rceil>b+c-x$ 
means $y>a+2b+2c-2x$.  Since $y<b+c-z$ this gives $a+2b+2c-2x<b+c-z$ which becomes $2x-z>a+b+c$.  
Adding this inequality to $x+z \geqslant b+c$ (which is equivalent to the assumption on hand) we obtain $3x>a+2b+2c$. 
At this point we suppose for the sake of contradiction that $4(L+z-x)+z > a+3b+3c$.  This means 
$(L+z-x+z)+3L+3z-3x>a+3b+3c$ which in light of the previous equation yields $(L+z-x+z)+3L+3z > 2a+5b+5c$. 
This in turn gives $3L+3z>2a+4b+4c$ since $L+z-x+z<b+c$.  Finally, we are left with $3z>-a+b+c$.  But at the 
same time $4(L+z-x)+z > a+3b+3c$ means $4(L+z-x+z)-3z > a+3b+3c$.  And again making use of $L+z-x+z<b+c$ 
this implies $3z<-a+b+c$, which is a contradiction. Hence we must have $4(L+z-x)+z \leqslant a+3b+3c$.
\item Condition \eqref{p4}: $(L+z-x)-2z \leqslant a$.  Again, $\mu$ is negative so we may consider two cases.  
First, if $z \geqslant \lceil \frac{y-a}{2}\rceil$ then $L+z-x \leqslant y$ implies $L+z-x-2z=(L+z-x)-y+a\leqslant a$.  
On the other hand if $z \geqslant b+c-x$ then $(L+z-x)-2z=L-x-z \leqslant L-b-c=a$.
\end{itemize}

Now since $P_{(L+z-x)z}$ is indeed a pseudohead, the facts that $L-\mathcal{E}(L+z-x,z)=L-(L-x)=x$ and 
$L+z-x \leqslant y <b+c-\mathcal{F}(L+z-z,z)$ (since the latter is equal to $b+c-z$) imply that  $\mu \in S(P_{(L+z-x)z})$
(in the second union in~\eqref{equation.SP union}).}

\item{ Now suppose $x+\mathcal{E}(y,z)\geqslant L$ and $y+z \geqslant b+c$.  
Note that this corresponds to Case 1(b)(ii) in Section~\ref{section.combinatorics}. 
To show uniqueness suppose $\mu \in S(P_{ij})$ with tail $T^{(L-p)(b+c-q)}$. 

If $\mu$ is from the first union in~\eqref{equation.SP union}, then $\mu = (x,i,j)$ for $i \leqslant x <p$. This means 
that $x+\mathcal{E}(y,z)=x+\mathcal{E}(i,j)<L$, contradicting our assumption.

If $\mu$ is from the second union in~\eqref{equation.SP union}, then $\mu=(p,y,j)$ for $i \leqslant y <q$. 
Thus $y<b+c-\mathcal{F}(i,j)$ which is equivalent to $y+j-(b+c)<\epsilon_{ij} \leqslant \epsilon_{yj}$ which implies 
$\epsilon_{yj}=0$ and $y+j-(b+c)<0$, contradicting $y+j=y+z \geqslant b+c$.

This means that $\mu$ can only come from the third union in~\eqref{equation.SP union}, so that we must have $x=p$ 
and $y=q$.  Hence $\mu$ can be in no other string than $S(T^{(L-x)(b+c-y)})$.

Now we show that $\mu \in S(T^{(L-x)(b+c-y)})$. First we need to check that $T^{(L-x)(b+c-y)}$ satisfies the tail 
conditions \eqref{t1}-\eqref{t3} for $E=L-x$ and $F=b+c-y$:
\begin{itemize}
\item Condition \eqref{t1}: $0 \leqslant F \leqslant c- \epsilon^{EF}$.  This means $0 \leqslant b+c-y \leqslant
c- \epsilon^{(L-x)(b+c-y)}$. The left-hand side follows from $y \leqslant b+c$.  The right-hand side says 
$\epsilon_{(L-x)(b+c-y)} \leqslant y-b$. We may assume $\epsilon^{(L-x)(b+c-y)}=a+2b+2c-x-2y$ because if 
it were $0$ then the fact that $z \leqslant c$ and $y+z \geqslant b+c$ imply $y-b \geqslant 0$ which would prove this side.
Under this assumption what we need to show becomes $x+3y \geqslant a+3b+2c$.  But $y+z \geqslant b+c$ implies that 
$\mathcal{E}(y,z)=3y+z-2(b+c)$, so the original assumption that $x+\mathcal{E}(y,z) \geqslant L$ becomes 
$x+3y+z \leqslant a+3b+3c$ which implies what we wanted to show as $z \leqslant c$.  
\item Condition \eqref{t2}: $2\epsilon^{EF} \leqslant E \leqslant F+a$. If $\epsilon^{EF}=0$ the left-hand side is immediate.  
Otherwise it is equivalent to $x+4y \geqslant a+3b+3c$.  This follows from $x+3y+z \leqslant a+3b+3c$ unless $y<c$.  
But this means we must have $z>b$ to obtain $y+z \geqslant b+c$.  Since $y \geqslant z$ this would give $b \leqslant c-2$
which is not allowed.  The right hand side follows directly from $x \geqslant y$. 
\item Condition \eqref{t3}: $4E+5F-3\epsilon^{EF} \leqslant a+3b+3c$.  This reduces to $4x+5y+3\epsilon^{(L-x)(b+c-y)} 
\geqslant 3a+6b+6c$.  If $\epsilon^{(L-x)(b+c-y)}=0$ we must have $x+2y \geqslant a+2b+2c$. Adding three times this 
inequality to the inequality $x-y \geqslant 0$ gives us what we desire.  If $\epsilon^{(L-x)(b+c-y)}>0$ then 
$\epsilon^{(L-x)(b+c-y)}=a+2b+2c-x-2y$ and the inequality $4x+5y+3\epsilon^{(L-x)(b+c-y)} \geqslant 3a+6b+6c$ 
reduces to $x-y \geq 0$.
\end{itemize}
Now we know that $T^{(L-x)(b+c-y)}$ is a valid tail. Denote $\Psi(L-x,b+c-y)=(i,j)$. In order to show that 
$\mu \in S(T^{(L-x)(b+c-y)})$ we need only verify that $j \leqslant z$.  That is to say $b+c-y+\epsilon^{(L-x)(b+c-y)} \leqslant z$.
If $\epsilon^{(L-x)(b+c-y)}=0$ this follows from the original assumption that $y+z \geqslant b+c$. Otherwise it reduces to 
$a+3b+3c \leqslant x+3y+z$.  But this is equivalent to the original assumption that $x+\mathcal{E}(y,z) \geqslant L$ 
since $y+z \geqslant b+c$ implies $\epsilon_{yz}=y+z-(b+c)$.}
\end{enumerate}
This concludes the proof of the first statement.

Now we prove the second statement.  Suppose $\mu=(x,y,z) \in S(P_{ij})$ for some pseudohead $P_{ij}$.  We must show 
that $z \geqslant \min(b+c-x,\lceil \frac{y-a}{2}\rceil)$.  We use two cases:
\begin{enumerate}
\item $\mu$ is in the first union in~\eqref{equation.SP union}. We show $z \geqslant \lceil \frac{y-a}{2}\rceil$. We have 
$y=i$ and $j=z$ so this becomes $j \geqslant \lceil \frac{i-a}{2}\rceil$.  But the latter is equivalent to $2j \geqslant i-a$ 
which is equivalent to condition \eqref{p4}.
\item $\mu$ is in the second or third union in~\eqref{equation.SP union}. We show $z \geqslant b+c-x$. First, if 
$\epsilon_{ij}>0$ then $i+j>b+c$ directly implies $j>b+c-x$ so that $z>b+c-x$. Now we assume $\epsilon_{ij}=0$. 
Since $x=L-\mathcal{E}(i,j)$ and $z\geqslant j$ it would be enough to show $j \geqslant b+c-(L-(i-j))$ which reduces 
to $j\geqslant -a+i-j$ but this follows from condition \eqref{p4}.
\end{enumerate}
\end{proof}

If $H_k^\ell \in H^+$, we define the \defn{appendage} associated to $H_k^\ell$ to be
\[
	A(H_k^\ell)=\{(k,\ell,z)\mid z < \min(b+c-k,\lceil \frac{\ell-a}{2}\rceil)\}.
\]

\begin{theorem} \label{pos}
Let $\mu \in \Lambda^+$. Then there exists unique  $H_k^\ell \in H^+$ such that $\mu \in A(H_k^\ell)$.  
Conversely, if $\mu \in A(H_k^\ell)$ for some positive head $H_k^\ell$, then $\mu \in \Lambda^+$.
\end{theorem}

\begin{proof}
Let $\mu=(x,y,z) \in \Lambda^+$.  Note that this correspond to Case 2 in Section~\ref{section.combinatorics}.
Then it is immediate that $\mu$ could only belong to the appendage $A(H_x^y)$.  
Since $z < \min(b+c-x,\lceil \frac{y-a}{2}\rceil)$ in particular $0 < \min(b+c-x,\lceil \frac{y-a}{2}\rceil)$.  This implies both 
$x <b+c$ and $y>a$ so (as $y \leqslant x$) $H_x^y \in H^+$.  Since $\mu$ is positive $z < \min(b+c-x,\lceil \frac{y-a}{2}\rceil)$,
so $\mu \in A(H_x^y)$.  

Now if $\mu=(x,y,z) \in A(H_k^\ell)$ for some head, then $x=k$ and $y=\ell$ and so the inequality 
$z < \min(b+c-x,\lceil \frac{y-a}{2}\rceil)$ is clearly satisfied implying that $\mu \in \Lambda^+$.
\end{proof}

%%%%%%%%%%%%%%%%%%%%%%%%%%%%%%%%%%%%%%%%%%%%%%%%%%%%%%%%%
\subsection{Chains}
\label{section.chains}
Suppose $T^{EF} \in T$.  Set $(i,j)=\Psi(E,F)$.  If $P_{ij} \in P^+$ set $(k,\ell)=\Theta(i,j)$.  We define the \defn{chain} of 
$T^{EF}$ to be 
\begin{equation}
	C(T^{EF})=\begin{cases}
	S(P_{ij}) & \text{if $P_{ij} \in P^-$,}\\
	S(P_{ij}) \cup A(H_k^\ell) & \text{if $P_{ij} \in P^+$.}
	\end{cases}
\end{equation} 

Our fundamental result concerning chains is the following.
\begin{theorem}
\label{theorem.chain decomp}
$\Lambda$ is the disjoint union:
\[
	\Lambda= \bigcup_{T^{EF} \in T} C(T^{EF}).
\]
Moreover, for each integer $m \in R^{EF}=[E+F,A-2E-3F+\max(\delta^{EF},\epsilon^{EF})]$ there is precisely one element 
$\mu \in C(T^{EF})$ with area $\mathsf{area}(\mu)=m$.
\end{theorem}

\begin{proof}
The first statement is immediate by combining Theorems~\ref{neg} and~\ref{pos}.  

Now fix $T^{EF}$ and set $(i,j)=\Psi(E,F)$.  If $P_{ij} \in P^-$, then by definition $C(T^{EF})=S(P_{ij})$. By construction, 
this string has one partition of area $m$ for each $m \in [\mathsf{area}(T^{EF}),\mathsf{area}(P_{ij})]$. But 
$\mathsf{area}(T^{EF})=E+F$.  Moreover, $\mathsf{area}(P_{ij})=A-2i-j=A-2E-3F + \epsilon^{EF}$ and since $P_{ij}\in P^-$ 
implies that $\epsilon_{ij}=\max(\delta_{ij},\epsilon_{ij})=\max(\delta^{EF},\epsilon^{EF})$ this means 
$\mathsf{area}(P_{ij})=A-2E-3F+\max(\delta^{EF},\epsilon^{EF})$.

Now suppose $P_{ij} \in P^+$.  Then $C(T^{EF})=S(P_{ij}) \cup A(H_k^\ell)$ has one partition of area $m$ for each 
$m \in [\mathsf{area}(T^{EF}),\mathsf{area}(P_{ij})]$ and one partition of area $n$ for each 
$n \in [\mathsf{area}(H_k^\ell)-\min(b+c-k,\lceil \frac{\ell-a}{2}\rceil)+1,\mathsf{area}(H_k^\ell)]$.  Again, $\mathsf{area}(T^{EF})=E+F$ 
and $\mathsf{area}(P_{ij})=A-2E-3F + \epsilon^{EF}$ so it suffices to prove the two equations
\begin{align*}
	&\mathsf{area}(H_k^\ell)-\min(b+c-k,\lceil \frac{\ell-a}{2}\rceil)+1=A-2E-3F + \epsilon^{EF}+1,\\
	&\mathsf{area}(H_k^\ell)=A-2E-3F+\max(\delta^{EF},\epsilon^{EF})=A-2E-3F+\delta^{EF}.
\end{align*}
However, we have
\begin{multline*}
	k+\ell=(i+j-\delta_{ij})+(i+\epsilon_{ij})=
	2i+j-(\delta_{ij}-\epsilon_{ij})\\
	=2(E+F-\epsilon^{EF})+(F+\epsilon^{EF})-(\delta^{EF}-\epsilon^{EF})=2E+3F-\delta^{EF},
\end{multline*}
so that $\mathsf{area}(H_k^\ell)=A-(k+1)=A-2E-3F+\delta^{EF}$ as desired. On the other hand
\begin{multline*}
	\min(b+c-k,\lceil \frac{\ell-a}{2}\rceil)=
	\min(b+c-(i+j-\delta_{ij},\lceil \frac{i+\epsilon_{ij}-a}{2}\rceil)\\
	=\min(b+c-(i+j)+\delta_{ij},\delta_{ij})=
	\delta_{ij} + \min(b+c-(i+j),0)\\ 
	=\delta_{ij}-\epsilon_{ij}=\delta^{EF}-\epsilon^{EF}.
\end{multline*}
Hence we have
\begin{multline*}
	\mathsf{area}(H_k^\ell)-\min(b+c-k,\lceil \frac{\ell-a}{2}\rceil)=A-2E-3F+\delta^{EF}-(\delta^{EF}-\epsilon^{EF})\\
	=A-2E-3F + \epsilon^{EF},
\end{multline*}
which gives the other equation we wanted after adding $1$ to both sides.
\end{proof}

Since $\Psi$ and $\Theta$ fix the area range, we can conclude the following statement.

\begin{corollary}
\label{corollary.cover}
Let $X$ represent the set of heads, the set of pseudoheads, or the set of tails. Then $\Lambda$ is the disjoint union 
of all chains which contain an element of $X$.  
Moreover, for $x \in X$ and each $m$ in the area range of $x$, there is precisely one element $\mu$ of area $m$ in 
the same chain as $x$.
\end{corollary}

%%%%%%%%%%%%%%%%%%%%%%%%%%%%%%%%%%%%%%%%%%%%%%%%%%%%%%%%%%
\subsection{Area preserving bijection between head and quasiheads}
\label{section.quasi}

We write $Q=Q_{\leqslant}^- \cup Q_>^- \cup Q^+$, where 
\begin{align*}
	Q_{\leqslant}^- &= \{Q\mms{st} \in Q\mid s+t \leqslant b+c, s \leqslant a\},\\
	Q_{>}^- &= \{Q\mms{st} \in Q\mid s+t > b+c\} \cup \{Q\mms{st} \in Q\mid s+t = b+c, s>a\},\\
	Q^+ &= \{Q\mms{st} \in Q\mid s+t < b+c, s > a\},
\end{align*}
and $H=P_{\leqslant}^- \cup P_>^- \cup H^+$, where
\[
	P_{\leqslant}^-=\{P_{ij} \in P^-\mid i+j \leqslant b+c\} \qquad \text{and} \qquad
	P_{>}^-=\{P_{ij} \in P^-\mid i+j > b+c\}.
\]

\begin{proposition}
\label{proposition:H-Q}
There is an area range preserving bijection from $H$ to $Q$.
\end{proposition}

\begin{proof}
We prove the proposition in three parts. First we show that the identity is an area range preserving bijection from 
$P_{\leqslant}^-$ to $Q_{\leqslant}^-$. Then we define an area range preserving bijection from $P_>^-$ to $Q_>^-$. 
Finally we define an area range preserving bijection from $H^+$ to $Q^+$.

\begin{enumerate}
%%%%%%%%%%%%%%%%%%%%%%%%%%%%%%%%%%%%%PART 1
%%%%%%%%%%%%%%%%%%%%%%%%% P<=Q<
\item The set $P_{\leqslant}^-$ is the set of triples $(i,i,j)$ obeying the conditions \eqref{p1}-\eqref{p4} as well as the 
inequalities $\delta_{ij} \leqslant \epsilon_{ij}$ and $i+j \leqslant b+c$. In light of~\eqref{p2}, $\epsilon_{ij}=0$ and 
$\delta_{ij} \leqslant \epsilon_{ij}$ simply becomes $i \leqslant a$.  But this in turn implies \eqref{p4}.  Moreover, 
adding $i \leqslant a$ to $3i+3j \leqslant 3b+3c$ gives condition \eqref{p3}.  Thus $P_{\leqslant}^-$ is the set of triples 
$(i,i,j)$ satisfying the four conditions in~\eqref{p1} and~\eqref{p2} as well as $i+j \leqslant b+c$ and $i \leqslant a$. 
On the other hand,  $Q_{\leqslant}^-$ is the set of triples $(s,s,t)$ satisfying the five conditions~\eqref{q1}-\eqref{q3} 
as well as $s+t \leqslant b+c$ and $s \leqslant a$. Since conditions \eqref{p1}-\eqref{p2} are equivalent to
\eqref{q1}-\eqref{q2}, if we can show that condition \eqref{q3} is implied by the other four conditions, it follows that 
$Q_{\leqslant}^-=P_{\leqslant}^-$.  Indeed, adding the three inequalities $s+t \leqslant b+c$, $s \leqslant a$, and 
$t \leqslant c$ gives $2s+2t \leqslant a+b+2c$.  This is strict unless we have equality in all of the three previous conditions. 
In particular, this would mean $t=c$ so that $m_{ct}=0$. Thus in any case $2s+2t \leqslant a+b+2c-m_{ct}$.  Therefore,
$Q_{\leqslant}^-=P_{\leqslant}^-$.  Since for $P_{ij} \in P_{\leqslant}^-$ $\max(0,\delta_{ij}-\epsilon_{ij})=0$, we have 
$R_{ij}=R\mms{st}$ if $i=s$, $j=t$, so that the identity is an area range preserving bijection between the two sets.

%%%%%%%%%%%%%%%%%%%%%%%%%%%%%%%%%%%%%%%%%%%%%
%%%%%%%%%%%%%%%%%PART 2
\item Let $\omega_{ij}=\max(0,\lceil\frac{2i+j-L}{2}\rceil)$ and define maps $\Phi$ and $\Phi^{-1}$ by
\[
	 \Phi(i,j)=(i+\omega_{ij},j-2\omega_{ij}) \qquad \text{and} \qquad  \Phi^{-1}(s,t)=(s-\omega_{st}, t+2\omega_{st}).
\]
Now if $\Phi(i,j)=(s,t)$ or $\Phi^{-1}(s,t)=(i,j)$, it is clear that $\omega_{ij}=\omega_{st}$.  From this it follows that 
$\Phi^{-1} \circ \Phi$ and $\Phi \circ \Phi^{-1}$ are the identity on $\mathbb{Z}^2$.
We claim that $\Phi$ induces an area range preserving bijection from $P_>^-$ to $Q_>^-$ via the rule that if 
$\Phi(i,j)=(s,t)$, then $(i,i,j) \mapsto (s,s,t)$ with the inverse  induced by $\Phi^{-1}$  via the rule that if $\Phi^{-1}(s,t)=(i,j)$,
then $(s,s,t) \mapsto (i,i,j)$.  

%%%%%%%%%%%%%%%%%%%%%%%%%% P> to Q>
First suppose $P_{ij} \in P_>^-$, so that conditions \eqref{p1}-\eqref{p4} are satisfied alongside 
$\delta_{ij}-\epsilon_{ij} \leqslant 0$ and $i+j>b+c$.  We need to check that if $\Phi(i,j)=(s,t)$, i.e., $s=i+\omega_{ij}$ 
and $t=j-2\omega_{ij}$, then $(s,t)$ satisfies conditions \eqref{q1}-\eqref{q3} and that $s+t \geqslant b+c$ and 
$(s+t=b+c) \implies s>a$. 

\begin{itemize}
%%%%%%%%%%%%%%%%q1
 \item Condition \eqref{q1}: $0 \leqslant t \leqslant c$.  This translates to $0 \leqslant j-2\omega_{ij} \leqslant c$. The right hand 
 side follows from $j \leqslant c$ (see the right hand side of \eqref{p1}).  If $\omega_{ij}=0$, the left hand side follows 
 from the left hand side of \eqref{p1}.  Otherwise it says that $j \geqslant 2\lceil\frac{2i+j-L}{2}\rceil$.  Now 
 $\delta_{ij}-\epsilon_{ij}\leqslant 0$ is equivalent to $i - \epsilon_{ij} \leqslant a$, but $i+j > b+c$ so $\epsilon_{ij}>0$ 
 and this  becomes $i-(i+j-(b+c)) \leqslant a$ or $-j \leqslant a-b-c$. Adding this to \eqref{p3} yields $4i \leqslant 2a+2b+2c$ 
 or $2i \leqslant L$.  This is enough to prove  $j \geqslant 2\lceil\frac{2i+j-L}{2}\rceil$ unless $2i=L$ and $j$ is odd.  But 
 then $4i+j$ is odd and $a+3b+3c$ is even so we have strictness in \eqref{p3}, that is, $4i+j < a+3b+3c$.  Hence adding this to 
 $-j \leqslant a-b-c$ results in $4i < 2a+2b+2c$ which contradicts $2i=L$.
 
 %%%%%%%%%%%%%%%%%%%q2
 \item Condition \eqref{q2}: $t \leqslant s \leqslant b+c$. This translates to $j-2\omega_{ij} \leqslant i+\omega_{ij} \leqslant b+c$.
 The left hand side follows from the left hand side of \eqref{p2}.  If $\omega_{ij}=0$, then the right hand side comes from the 
 right hand side \eqref{p2}.  Otherwise the right hand side says $i+\lceil\frac{2i+j-L}{2}\rceil \leqslant b+c$ which follows from 
 $i+\frac{2i+j-L}{2} \leqslant b+c$ (which is equivalent to \eqref{p3}) since $b+c$ is an integer.
 
 %%%%%%%%%%%%%%%%%%q3
 \item Condition \eqref{q3}: $2s+2t \leqslant a+b+2c-m_{ct}$.  This translates to $2i+2j-2\omega_{ij}\leqslant a+b+2c-m_{cj}$ 
 (note that $m_{c(j-2\omega_{ij})}=m_{cj}$). If $\omega_{ij}=0$ then $2i+j \leqslant L$ so it suffices to show 
 $j \leqslant c-m_{cj}$ which is evident by the definition of $m_{cj}$ and $j \leqslant c$. On the other hand if 
 $\omega_{ij}>0$, then proving $2i+2j \leqslant (2i+j-L)+ a+b+2c-m_{cj}$ suffices since $(2i+j-L) \leqslant 2\omega_{ij}$.  
 But the former again reduces to  $j \leqslant c-m_{cj}$.
 
 %%%%%%%%%%%%%%%%%s+t>=b+c
 \item $s+t \geqslant b+c$.  This says $i+j-\omega_{ij} \geqslant b+c$.  This is clear from the definition of $P_>^-$ if 
 $\omega_{ij}=0$ so suppose $\omega_{ij}>0$.  Now, as in the first bullet point, $\delta_{ij}-\epsilon_{ij} \leqslant 0$ 
 and $i+j>b+c$ imply $j \geqslant b+c-a$. The latter is equivalent to $2i+2j-2\frac{2i+j-L}{2} \geqslant 2b+2c$, or, 
 dividing by $2$, $i+j-\frac{2i+j-L}{2} \geqslant b+c$.  But since $b+c$ is an integer, this implies 
 $i+j-\lceil\frac{2i+j-L}{2}\rceil \geqslant b+c$ as desired.

%%%%%%%%%%%%%%%s+t=b+c --> s>a
\item $(s+t=b+c) \implies s>a$.  This translates to, if $i+j-\omega_{ij}=b+c$, then $i+\omega_{ij}>a$.  If $\omega_{ij}=0$,
the hypothesis would clearly contradict the assumption $i+j>b+c$.  Thus we may assume $\omega_{ij}>0$ which 
means $2i+j > L$. Adding this to $-i-j \geqslant -b-c - \omega_{ij}$  gives $i > a-\omega_{ij}$ as desired.
\end{itemize}

%%%%%%%%%%%%%%%%%%%%%%%%%%%%%%%Q> to P>
Now suppose $Q\mms{st} \in Q_>^-$, so that conditions \eqref{q1}-\eqref{q3} are satisfied alongside  $s+t \geqslant b+c$ 
and $(s+t=b+c) \implies s>a$.   We need to check that if $\Phi^{-1}(s,t)=(i,j)$, that is, $i=s-\omega_{st}$ and 
$j=t+2\omega_{st}$, then $(i,j)$ satisfies conditions \eqref{p1}-\eqref{p3} as well as $\delta_{ij}-\epsilon_{ij} \leqslant 0$ 
and $i+j>b+c$. (We do not need to check condition \eqref{p4} as adding $-3i-3j<-3b-3c$ to \eqref{p3} yields $i-2j<a$.) 

\begin{itemize}
%%%%%%%%%%%%%p1
 \item Condition \eqref{p1}: $0 \leqslant j \leqslant c$.  This translates to $0 \leqslant t+2\omega_{st} \leqslant c$.  
 The left hand side follows from $0 \leqslant t$ (which is the left hand side of \eqref{q1}).  Now, if either $t-c$ or $a+b$ 
 are odd, then $2s+2t \leqslant a+b+2c-m_{ct}$ implies $2s+2t <a+b+2c$, which is to say $t+2\frac{2s+t-L}{2}<c$ 
 so that $t+2\omega_{st} \leqslant c$. On the other hand if both $t-c$ and $a+b$ are even, then we can only 
 deduce $t+2\frac{2s+t-L}{2}\leqslant c$ from \eqref{q3}, but in this case $\frac{2s+t-L}{2}=\omega_{st}$ so we still get 
 what we want.
 
 %%%%%%%%%%%%%%%%%p2
 \item Condition \eqref{p2}: $j \leqslant i \leqslant b+c$. This translates to $t+2\omega_{st} \leqslant s-\omega_{st} 
 \leqslant b+c$. The right hand side follows from the right hand side of \eqref{q2}.  If $\omega_{st}=0$, then the left 
 hand side comes from the left hand side \eqref{q2}.  Now suppose $\omega_{st}>0$. We need to show that 
 $t+2\omega_{st} \leqslant s-\omega_{st}$.  The inequality we wish to show is equivalent to 
 $2t-2s+6\omega_{st} \leqslant 0$. Since $2\omega_{st}$ can be rewritten as $2s+t-L+m_{Lt}$ this becomes 
 $4s+5t \leqslant 3L-3m_{Lt}$.
 
 First suppose that $s+t>b+c$ and $m_{Lt} \leqslant m_{ct}$. Since twice \eqref{q3} reads $4s+4t \leqslant 2L +2c -2m_{ct}$ 
 it suffices to prove $t \leqslant a+b-c+2m_{ct}-3m_{Lt}$, since the sum of the last two inequalities mentioned gives the 
 one at the end of the last sentence. Since $t \leqslant c$, it suffices to show $a+b \geqslant 2c-2m_{ct}+3m_{Lt}$.  
 Since  $s+t >b+c$, we have $-2s-2t\leqslant -2b-2c-2$ which we can add to \eqref{q3} to get  $a \geqslant b+2+m_{ct}$ 
 or  $a - b \geqslant 2+m_{ct}$. Adding this to $2b \geqslant 2c-2$ yields $a+b \geqslant 2c+ m_{ct}$ which implies 
 $a+b \geqslant 2c-2m_{ct}+3m_{Lt}$ since we are assuming $m_{Lt} \leqslant m_{ct}$.
 
 Now suppose that $s+t>b+c$, but $0=m_{Lt} < m_{ct}=1$. Since $L+c$ must be odd, we have strictness in \eqref{q3}, 
 that is we have $2s+2t < a+b+2c=L+c$.  Thus we have $4s+4t \leqslant 2L +2c -2$  so it suffices to prove 
 $t \leqslant a+b-c-1$ since adding these gives $4s+5t \leqslant 3L-3m_{Lt}$ as desired. Again, $t \leqslant c$ 
 so its enough to show $a+b \geqslant 2c+1$.  Since strictness of \eqref{q3} implies $2s+2t \leqslant a+b+2c-1$,
 we can add this to $-2s-2t\leqslant -2b-2c-2$ to obtain $a \geqslant b+3$. Since $b \geqslant c-1$, this implies 
 $a+b \geqslant 2c+1$ as desired.
 
 Finally suppose that $s+t=b+c$ (and so also $s>a$).  We need to show $4(b+c)+t \leqslant 3L-3m_{Lt}$, or equivalently, 
 $t \leqslant 3a-b-c -3m_{Lt}$. When $s+t=b+c$, condition \eqref{q3} becomes $a \geqslant b+m_{ct}$.  Thus $s>a$
 implies $s \geqslant b+m_{ct}+1$, which in turn means $t \leqslant c-m_{ct}-1$.  In fact $t \leqslant c-2$ because if 
 $t=c-1$ then we would have $m_{ct}=1$ and thus $t \leqslant c-2$.
Now since $a \geqslant b+m_{ct} \geqslant c-1+m_{ct}$ this means $t \leqslant c-2 + (a-b-m_{ct}) + 2(a-c+1-m_{ct})$ 
or that $t \leqslant 3a-b-c-3m_{ct}$ which implies $t \leqslant 3a-b-c -3m_{Lt}$ unless $0=m_{ct}<m_{Lt}=1$. But in the 
latter case $a+b$ must be odd so $a \geqslant b$ implies $a>b$. Thus we have $a \geqslant b+1 \geqslant c$ so that 
$t \leqslant c-2 + (a-b-1) + 2(a-c)$ or $t \leqslant 3a-b-c -3$.

%%%%%%%%%%%%%%%%%p3
\item Condition \eqref{p3}: $4i+j \leqslant a+3b+3c$. This translates to $4s+t-2\omega_{st} \leqslant a+3b+3c$.  
If $\omega_{st}=0$, then $2s+t \leqslant L$ which we add to two times the right hand side of \eqref{q2} to obtain
$4s+t\leqslant a+3b+3c$ as desired.  Now suppose $\omega_{st}>0$. Then $4s+t-2\omega_{st}=
4s+t-2\lceil \frac{2s+t-L}{2} \rceil \leqslant 4s+t-2\frac{2s+t-L}{2}=2s+L \leqslant L+b+c$, where the last inequality 
comes from the right hand of \eqref{q2}.

%%%%%%%%%%%%%%%%% del <= eps
\item $\delta_{ij} - \epsilon_{ij} \leqslant 0$.  This reduces to $i-\epsilon_{ij} \leqslant a$ which says that 
$s-\omega_{st}-\max(0,s+t+\omega_{st}-(b+c)) \leqslant a$.  Since $s+t \geqslant b+c$, this is equivalent to 
$-t-2\omega_{st} \leqslant a-b-c$.  If $\omega_{st}=0$ then $2s+t \leqslant L$ and which we may add to 
$-2s-2t\leqslant -2b-2c$ to get $-t \leqslant a-b-c$ as desired.  If $\omega_{st}>0$ then it will suffice to show 
$-t-2\frac{2s+t-L}{2} \leqslant a-b-c$ but this reduces to $-2s-2t \leqslant -2b-2c$.

 %%%%%%%%%%%%%% i+j>b+c
 \item $i+j>b+c$.  This means $s+t+\omega_{st}>b+c$.  If $\omega_{st}>0$ this follows from $s+t \geqslant b+c$.  
 Now suppose $\omega_{st}=0$. Then $2s+t \leqslant L$ which we may add to $-s-t \leqslant -b-c$ to get $s \leqslant a$.  
 Thus the assumptions $s+t \geqslant b+c$ and $(s+t=b+c) \implies s>a$ reveal that $s+t>b+c$.
\end{itemize}

This proves that $\Phi$ induces a bijection from $P_>^-$ to $Q_>^-$. Moreover if $P_{ij} \in P_>^-$, then 
$\delta_{ij}\leqslant \epsilon_{ij}$ and $i+j > b+c$, so the area range reduces to $R_{ij}=[2i+j-(b+c),A-2i-j]$.  
Now if $\Phi(i,j)=(s,t)$, then since by the above $Q\mms{st} \in Q_>^-$ we have $s+t \geqslant b+c$ so that we 
get $R\mms{st} = [2s+t-(b+c),A-2s-t]$.  Since it is clear that $2i+j=2s+t$ we see that $R_{ij}=R\mms{st}$ and so the 
bijection induced by $\Phi$ preserves the area range.

%%%%%%%%%%%%%%%%%%%%%%%%%%%%%%%%%%%%%%%%%%%%%%%%%%%PART 3

\item Define maps $\Omega$ and $\Omega^{-1}$ by
\[
	 \Omega(k,\ell)=(\ell,k-\ell) \qquad \text{and} \qquad  \Omega^{-1}(s,t)=(s+t,s).
\]
Since $\Omega$ is an invertible linear transformation (with inverse $\Omega^{-1}$), it is immediate that 
$\Omega^{-1} \circ \Omega$ and $\Omega \circ \Omega^{-1}$ are the identity on $\mathbb{Z}^2$.  We 
claim that $\Omega$ induces an area range preserving bijection from $H^+$ to $Q^+$ via the rule that, if 
$\Omega(k,\ell)=(s,t)$, then $(k,\ell,0) \rightarrow (s,s,t)$ with inverse  induced by $\Omega^{-1}$  via the rule that,
if  $\Omega^{-1}(s,t)=(k,\ell)$ then $(s,s,t) \rightarrow (k,\ell,0)$.

First suppose that $H_k^\ell \in H^+$ so that $a < \ell \leqslant k <b+c$, and that $\Omega(k,\ell)=(s,t)$, that is, 
$s=\ell$ and $t=k-\ell$. We need to check inequalities \eqref{q1}-\eqref{q3} as well as $s+t<b+c$ and $s>a$.

\begin{itemize}
\item Condition \eqref{q1}: $0 \leqslant t \leqslant c$.  This translates to $0 \leqslant k-\ell \leqslant c$.  The left hand 
side is true because $\ell\leqslant k$.  Moreover, since $\ell>a \geqslant b-1$, the inequalities $\ell \geqslant b$ and 
$k < b+c$ give $k-\ell<c$.

\item Condition \eqref{q2}: $t \leqslant s \leqslant b+c$.  This translates to $k-\ell \leqslant \ell \leqslant b+c$.  Now 
$k \leqslant b+c-1 \leqslant 2a+1 <2\ell$ which established the left hand side. On the other hand $\ell \leqslant k <b+c$ 
makes the right hand side obvious.

\item Condition \eqref{q3}: $2s+2t \leqslant a+b+2c-m_{ct}$.  This translates to $2k \leqslant a+b+2c-m_{c(k-\ell)}$.  
But $k \leqslant b+c-1$ and since $a \geqslant b-1$ also $k \leqslant a+c$.  Since adding these gives $2k \leqslant a+b+2c-1$
we are done.

\item $s+t < b+c$.  This translates to $\ell+k-\ell<b+c$, that is, $k <b+c$, as assumed.

\item $s>a$. This says $\ell>a$, as assumed.

\end{itemize}

Now suppose that $Q\mms{st} \in Q^+$ so that the inequalities \eqref{q1}-\eqref{q3} hold and we have $s+t<b+c$ and 
$s>a$. We need to show that if $\Omega^{-1}(s,t)=(k,\ell)$ then $H_k^\ell \in H^+$, that is, $a < \ell \leqslant k <b+c$.  
Since $k=s+t$ and $\ell=s$, this says $a < s \leqslant s+t <b+c$.  The left and right hand inequalities are those 
assumed above.  The middle inequality follows from the left hand side of \eqref{q1}.
\end{enumerate}

This proves that $\Omega$ induces a bijection from $H^+$ to $Q^+$. Moreover if $H_k^\ell \in H^+$, then 
$R_k^\ell=[\ell,A-k-\ell]$.  Now if $\Omega(i,j)=(s,t)$, then since by the above $Q\mms{st} \in Q^+$, we have $s+t < b+c$ 
so $\epsilon_{st}=0$ and $R\mms{st} = [s,A-2s-t]=[\ell,A-2\ell-k+\ell]$.  Thus $R_k^\ell=R\mms{st}$ and so the bijection 
induced by $\Omega$ preserves the area range.
\end{proof}

%%%%%%%%%%%%%%%%%%%%%%%%%%%%%%%%%%%%%%%%%%%%%%%%%%%%%%%%%%
\subsection{Combinatorial recursion}
\label{section.combinatorial recursion}

In this section, we show that the combinatorial expression of Theorem~\ref{theorem.symmetric chains} also satisfies
the recursion relations of Lemma~\ref{lemma.two step recurrence} for $c\geqslant 1$ and equals $F(a,b,0)$
and $F(a,b,-1)$ for $c=0$ and $c=-1$, respectively.

Recall that the set of quasiheads is defined as
\[
	\widetilde{Q}(a,b,c)=\{(i,j)\ |\ 0\leqslant j\leqslant c,j\leqslant i\leqslant b+c, 2i+2j\leqslant a+b+2c-m_{cj}\}
\]
where $m_{cj}=c-j \pmod{2}$. Define
\begin{equation}
\label{equation.H comb}
	H_{\comb}(a,b,c)=\sum_{(i,j)\in\widetilde{Q}(a,b,c)}q^{A-2i-j}t^{i+\epsilon_{ij}},
\end{equation}
where $\epsilon_{ij}=\max(0,i+j-b-c)$.

\begin{lemma}
\label{lem: recursion H comb}
For $a+1 \geqslant b$, $a+1,b+1 \geqslant c\geqslant 1$, we have
\begin{multline}
	H_{\comb}(a,b,c) = H_{\comb}(a+2,b+2,c-2) + (qt)^{c}H(a+c,b-c)\\
	+(qt)^{c-1}H(a+c,b-c+2)+
	\sum_{2\leqslant \ell \leqslant \min(2c,a-b)} q^{a+2c-\ell}t^{\ell+b}
	-\delta_{a,b-1}q^{a+2c}t^{b}\\
	-(\delta_{a,b}+\delta_{a,b-1})q^{a+2c-1}t^{b+1},
\end{multline}
where $H(a,b)$ is given by \eqref{H for n=3}.
\end{lemma}

\begin{proof}
Observe that if $(a',b',c')=(a+2,b+2,c-2)$, then 
$b'+c'=b+c$, $a'+b'+2c'=a+b+2c$, $m_{c'j}=m_{cj}$.
Therefore
\[
	\widetilde{Q}(a',b',c')=\{(i,j) \mid 0\leqslant j\leqslant c',j\leqslant i\leqslant b+c, 2i+2j\leqslant a+b+2c-m_{cj}\}.
\]
We conclude that $\widetilde{Q}(a',b',c')\subseteq \widetilde{Q}(a,b,c)$
and the difference of these two sets consists of 
$(i,j)\in \widetilde{Q}(a,b,c)$ with $j=c$ or $j=c-1$.
In the former case, the inequalities have the form
\begin{equation}
\label{eq: j=c}
	c\leqslant i\leqslant b+c, \qquad 2i\leqslant a+b
\end{equation}
and the contribution to $H_{\comb}$ equals 
\[
	\sum_{\substack{c\leqslant i\leqslant b+c\\ 2i\leqslant  a+b}} q^{A-c-2i}\; t^{i+\max(0,i-b)}.
\]
This sum breaks into two parts for $c\leqslant i\leqslant b$ and for $b+1 \leqslant i$
\begin{equation*}
	\sum_{\substack{c\leqslant i\leqslant b\\ 2i\leqslant a+b}} q^{a+2b+2c-2i} \; t^{i} 
	+ \sum_{\substack{b+1\leqslant i\leqslant b+c\\ 2i\leqslant a+b}} q^{a+2b+2c-2i} \; t^{2i-b}.
\end{equation*}
If $a \geqslant b$, the restriction $2i\leqslant a+b$ in the first sum is redundant and so it becomes $(qt)^c H(a+c,b-c)$.  
On the other hand if $a=b-1$, the first sum does not contain the $i=b$ term $q^{a+2c}t^b$ but  $(qt)^c H(a+c,b-c)$ does.  
Thus we conclude the above is equal to 
\begin{equation*}
	(qt)^c H(a+c,b-c)-\delta_{a,b-1}q^{a+2c}t^b + \sum_{2b+2\leqslant 2i\leqslant \min(2b+2c,a+b)} q^{a+2b+2c-2i}\; t^{2i-b}.
\end{equation*}

Similarly, in the case $j=c-1$ for $a>b$ we obtain
\[
	c-1\leqslant i\leqslant b+c, \qquad 2i\leqslant a+b+1
\]
and the contribution to $H_{\comb}$ equals 
\begin{equation*}
	\sum_{\substack{c-1\leqslant i\leqslant b+1\\ 2i\leqslant a+b+1}} q^{a+2b+2c-2i+1}\; t^{i}
	+\sum_{\substack{b+2\leqslant i\leqslant b+c\\ 2i\leqslant a+b+1}}q^{a+2b+2c-2i+1}\; t^{2i-1-b}.
\end{equation*}
If $a > b$ the restriction $2i\leqslant a+b+1$ in the first sum is redundant and so it becomes $(qt)^{c-1} H(a+c,b-c+2)$.  
On the other hand if $a=b$ or $a=b-1$ the first sum does not contain the $i=b+1$ term $q^{a+2c-1}t^{b+1}$ or the $i=b$ 
term $q^{a+2c}t^b$ but  $(qt)^{c-1} H(a+c,b-c-2)$ does.  Thus we conclude the above is equal to 
\begin{multline*}
	(qt)^{c-1}H(a+c,b-c+2)-(\delta_{a,b-1}+\delta_{a,b})q^{a+2c-1}t^{b+1}\\
	+\sum_{2b+3\leqslant 2i-1\leqslant \min(2b+2c-1,a+b)}q^{a+2b+2c-2i+1}t^{2i-1-b}.
\end{multline*}

Finally,
\begin{multline*}
	\sum_{2b+2\leqslant 2i\leqslant \min(2b+2c,a+b)}q^{a+2b+2c-2i}t^{2i-b}
	+\sum_{2b+3\leqslant 2i-1\leqslant \min(2b+2c-1,a+b)}q^{a+2b+2c-2i+1}t^{2i-1-b}\\
	= \sum_{2b+2\leqslant k\leqslant \min(2b+2c,a+b)}q^{a+2b+2c-k}t^{k-b},
\end{multline*}
where we combined terms with even and odd $k$.
If we denote $\ell=k-2b$, then
\[
	\sum_{2b+2\leqslant k\leqslant \min(2b+2c,a+b)}q^{a+2b+2c-k}t^{k-b}=
	\sum_{2\leqslant \ell\leqslant \min(2c,a-b)}q^{a+2c-\ell}\; t^{\ell+b}.
\]
\end{proof}

\begin{corollary}
\label{corollary.Fcomb recursion}
Let 
\[
	F_{\comb}(a,b,c)=\frac{1}{1-t/q}H_{\comb}(a,b,c;q,t)+\frac{1}{1-q/t}H_{\comb}(a,b,c;t,q).
\]
Then for $a+1 \geqslant b$, $a+1,b+1 \geqslant c\geqslant 1$ we have 
\begin{multline}
	F_{\comb}(a,b,c)=F_{\comb}(a+2,b+2,c-2)+(qt)^cF(a+c,b-c)\\+(qt)^{c-1}F(a+c,b-c+2)+
	\sum_{2\leqslant \ell\leqslant \min(2c,a-b)}(qt)^{\ell+b}F(a-b+2c-2\ell)\\
	- \sum_{j=a-b+1}^1 (qt)^{b+j} F(a-b+2c-2j).
\end{multline}
\end{corollary}

\begin{proof}
This follows directly from Lemma~\ref{lem: recursion H comb}, using~\eqref{H to F}, and Example~\ref{example.n=2}.
Also note that
\begin{multline*}
	\sum_{j=a-b+1}^1 (qt)^{b+j} F(a-b+2c-2j) = \delta_{a,b} (qt)^{b+1} F(a-b+2c-2) \\
	+ \delta_{a,b-1} \left[ (qt)^b F(a-b+2c) + (qt)^{b+1} F(a-b+2c-2) \right].
\end{multline*}
\end{proof}

We need to check the base cases.

\begin{lemma}
\label{lemma.initial condition}
We have
\begin{equation*}
\begin{split}
	F_{\comb}(a,b,0)&=F(a,b,0) \quad \text{for $a,b\geqslant 0$,}\\
	F_{\comb}(a,b,-1)&=F(a,b,-1) \quad \text{for $a,b \geqslant 1$.}
\end{split}
\end{equation*}
\end{lemma}

\begin{proof}
For $a,b\geqslant 0$ and $c=0$, we have $j=0$ in $\widetilde{Q}(a,b,0)$, so $0\leqslant i\leqslant b$. Therefore, by 
comparing~\eqref{equation.H comb} with~\eqref{H for n=3}
\[
	H_{\comb}(a,b,0)=H(a,b),
\]
and hence $F_{\comb}(a,b,0)=F(a,b)$. Furthermore, by Corollary~\ref{cor:set 0} the first claim follows.

For $a,b\geqslant 1$, we have by Lemma~\ref{lemma.-1} and the fact that $F_{\comb}(a,b,-1)=0$ by definition
that $F(a,b,-1)=F_{\comb}(a,b,-1)=0$.
\end{proof}

\begin{corollary}
\label{corollary.F=Fcomb}
For nonnegative integers $a,b,c$ and $a+1 \geqslant b$, $a+1,b+1 \geqslant c$, we have $F(a,b,c) = F_{\comb}(a,b,c)$ 
proving Theorem~\ref{theorem.symmetric chains}.
\end{corollary}

\begin{proof}
By Lemma~\ref{lemma.two step recurrence} and Corollary~\ref{corollary.Fcomb recursion}, $F(a,b,c)$ and 
$F_{\comb}(a,b,c)$ satisfy the same two step recursion. Hence the equality $F(a,b,c) = F_{\comb}(a,b,c)$ 
can be reduced to the equalities $F(a,b,0) = F_{\comb}(a,b,0)$ for $a,b\geqslant 0$ and 
$F(a,b,-1) = F_{\comb}(a,b,-1)$ for $a,b\geqslant 1$. These are given in Lemma~\ref{lemma.initial condition}.
\end{proof}

%%%%%%%%%%%%%%%%%%%%%%%%%%%%%%%%%%%%%%%%%%%%%%%%%%%%%%%%%
\subsection{Proof of Theorem~\ref{theorem.combinatorics}}
\label{section.proof comb}

By Corollary~\ref{corollary.F=Fcomb}, we have that $F(a,b,c)=F_{\comb}(a,b,c)$. By Proposition~\ref{proposition:H-Q},
there is an area preserving bijection between quasiheads and heads. Combined with 
Corollary~\ref{corollary.cover}, there is also an area preserving bijection with pseudoheads and tails.
Furthermore, each $\lambda \subseteq \lambda(a,b,c)$ sits in precisely one chain indexed by
a given pseudohead (or head). The proofs of Theorems~\ref{neg} and~\ref{pos} tell us, which chain $\lambda$ sits in
depending on the cases spelled out in Section~\ref{section.combinatorics}:

\begin{center}
\begin{tabular}{l|l}
Case & Chain membership\\
\hline
Case 1(a) & $\lambda \in C(P_{yz})$\\
Case 1(b)(i) & $\lambda \in C(P_{(L+z-x)z})$\\
Case 1(b)(ii) & $\lambda \in C(T^{(L-x)(b+c-y)})$\\
Case 2 & $\lambda \in C(H_x^y)$
\end{tabular}
\end{center}
Now if the area range for a given chain is $[r,R]$, then due to the symmetry between $q$ and $t$ in each chain,
we have
\[
	\mathsf{stat}(\lambda) = r + R - \mathsf{area}(\lambda) = r + R -A + x+y+z
\]
for $\lambda=(x,y,z)$. Using the area ranges for pseudoheads, tails, and heads as given in Table~\ref{table.indexing sets}
and the beginning of this section, this yields \eqref{equation.statistics}.
In Case 1(a), we first obtain $\mathsf{stat}(\lambda)=x+\max(\epsilon_{yz},\delta_{yz})$, which is equal to
$x+\max(0,\lceil \frac{y-a}{2} \rceil, y+z-b-c,\lceil \frac{2y+z-L}{2} \rceil)$.
In Case 1(b)(i), we first obtain $\mathsf{stat}(\lambda)=-L+2x+y-z+\max(\epsilon_{(L+z-x)z},\delta_{(L+z-x)z})$,
but using that $y+z\leqslant b+c$ and $L+z-x\leqslant y$, we obtain $\epsilon_{(L+z-x)z}=0$ and 
$\delta_{(L+z-x)x}= \lceil \frac{L+z-x-a}{2} \rceil$. Combined with Theorem~\ref{theorem.symmetric chains}
this proves Theorem~\ref{theorem.combinatorics}.

%%%%%%%%%%%%%%%%%%%%%%%%%%%%%%%%%%%%%%%%%%%%%%%%%%%%%%%%%
\bibliographystyle{alpha}
\bibliography{main}{}
%%%%%%%%%%%%%%%%%%%%%%%%%%%%%%%%%%%%%%%%%%%%%%%%%%%%%%%%%

\end{document}